\documentclass[11pt,leqno]{article}
\RequirePackage[leqno,cmex10]{amsmath}
\RequirePackage[arxiv,noinfoline]{imsart}
\RequirePackage[OT1]{fontenc}
\RequirePackage{amsthm,amssymb}
\usepackage{mathrsfs,enumerate,graphicx,dsfont,hyperref}
\usepackage[initials]{amsrefs}
\usepackage[letterpaper,margin=2.5cm]{geometry}
\usepackage{enumerate,stmaryrd}

\newcommand{\ws}[1][1.5]{
  \mathrel{\overset{}{\scalebox{#1}[1]{$\sim$}}}
}

\newenvironment{enumeratea}{\begin{enumerate}[\upshape (a)]\setlength{\itemsep}{4pt}}{\end{enumerate}}

\numberwithin{equation}{section}
\numberwithin{figure}{section}
\numberwithin{table}{section}
\theoremstyle{plain}
\newtheorem{thm}{Theorem}[section]
\newtheorem{lem}[thm]{Lemma}

\newtheorem{prop}[thm]{Proposition}

\theoremstyle{definition}

\newtheorem{pro}[thm]{Property}

\theoremstyle{remark}
\newtheorem{rem}[thm]{Remark}
\newcommand{\lrs}{\leftrightsquigarrow}
\renewcommand{\le}{\leqslant}
\renewcommand{\ge}{\geqslant}

\newcommand{\rrb}{\rrbracket}
\newcommand{\llb}{\llbracket}
\newcommand{\ra}{\rangle}
\newcommand{\la}{\langle}
\newcommand{\wt}{\widetilde}

\newcommand{\ind}{\mathds{1}}
\newcommand{\eps}{\varepsilon}

\newcommand{\norm}[1]{\left\Vert#1\right\Vert}

\newcommand{\ie}{\emph{i.e.,}\ }

\newcommand{\iid}{\emph{i.i.d.}~}

\def\qed{\hfill $\blacksquare$}
\newcommand{\cGU}{\cG^{\text{unknown}}}

\newcommand{\vpi}{\boldsymbol{\pi}}
\newcommand{\vPi}{\boldsymbol{\Pi}}
\newcommand{\vgY}{\boldsymbol{\gY}}
 \let\gb=\beta \let\gc=\gamma  
     \let\gl=\lambda

\let\gC=\Gamma \let\gD=\Delta \let\gF=\Phi  
           
\let\gY=\Upsilon
\newcommand{\fp}{\mathfrak{p}}

\newcommand{\cB}{\mathcal{B}}

\newcommand{\cG}{\mathcal{G}}

\newcommand{\vA}{\mathbf{A}}\newcommand{\vB}{\mathbf{B}}

\newcommand{\vI}{\mathbf{I}}
\newcommand{\vM}{\mathbf{M}}

\newcommand{\vR}{\mathbf{R}}
\newcommand{\vV}{\mathbf{V}}
\newcommand{\vX}{\mathbf{X}}\newcommand{\vZ}{\mathbf{Z}}

\newcommand{\vi}{\mathbf{i}}

\newcommand{\vn}{\mathbf{n}}
\newcommand{\vp}{\mathbf{p}}
\newcommand{\vu}{\mathbf{u}}
\newcommand{\vv}{\mathbf{v}}\newcommand{\vw}{\mathbf{w}}\newcommand{\vx}{\mathbf{x}}
\newcommand{\vy}{\mathbf{y}}\newcommand{\vz}{\mathbf{z}}

\newcommand{\dN}{\mathds{N}}

\newcommand{\dR}{\mathds{R}}
\newcommand{\dZ}{\mathds{Z}}
\newcommand{\dL}{\mathds{L}}

\newcommand{\sE}{\mathscr{E}}
\newcommand{\sG}{\mathscr{G}}

\newcommand{\sP}{\mathscr{P}}

\newcommand{\sV}{\mathscr{V}}
\newcommand{\sH}{\mathscr{H}}
\newcommand{\sZ}{\mathscr{Z}}
\DeclareMathOperator{\cov}{\mathds{C}ov}
\DeclareMathOperator{\var}{\mathds{V}ar}
\DeclareMathOperator{\E}{\mathds{E}}
\DeclareMathOperator{\pr}{\mathds{P}}

\def\convD{\,{\buildrel d \over \longrightarrow}\,}
\def\convP{\,{\buildrel P \over \longrightarrow}\,}
\def\eqd{\,{\buildrel d \over =}\,}

\def\beq{ \begin{equation} }
 \def\eeq{ \end{equation} }
 \def\beqx{ \begin{equation*} }
 \def\eeqx{ \end{equation*} }
 \def\beqa{\begin{eqnarray}}
 \def\eeqa{\end{eqnarray}}
 \def\beqax{\begin{eqnarray*}}
 \def\eeqax{\end{eqnarray*}}
\def\eqd{\,{\buildrel d \over =}\,}

\def\corS{}
\def\corO{}
\def\corOO{}
\begin{document}

\begin{frontmatter}
\title{Thresholds for Detecting an Anomalous Path from Noisy Environments}
\runtitle{Detection Threshold for Anomalous Path}

\begin{aug}
\author{\fnms{Shirshendu} \snm{Chatterjee}\thanksref{t1}\ead[label=e1]{shirshendu@ccny.cuny.edu}  \ead[label=u1,url]{http://mathsci.ccnysites.cuny.edu/}}
\and
\author{\fnms{Ofer} \snm{Zeitouni}\thanksref{t2}\ead[label=e2]{ofer.zeitouni@weizmann.ac.il}
\ead[label=u2,url]{http://www.wisdom.weizmann.ac.il/{\raise.17ex\hbox{$\scriptstyle\sim$}}zeitouni/}}

\thankstext{t1}{Partially Supported by Simons Foundation Collaborative Research Funds.}
\thankstext{t2}{Funded in part by an Israel Science Foundation grant.}

\runauthor{Chatterjee and Zeitouni}

\affiliation{City University of New York, City College
  \thanksmark{t1}
  and 
Weizmann Institute of Science and New York University
\thanksmark{t2}
}

\address{Department of Mathematics \\City University of New York, City College \\ 160 Convent Ave,  New York, NY 10031 , USA\\
\printead{e1}\\
\printead{u1}}
\address{Department of Mathematics \\ Weizmann Institute of Science,\\ POB 26, Rehovot 76100, Israel\\
\printead{e2}\\
\printead{u2}}
\end{aug}

\date{\today}
\begin{abstract}
  \corO{
We consider the ``searching for a trail in a maze'' 
composite hypothesis testing problem, in which one
attempts to detect an anomalous 
 directed
path in a lattice 2D box of side $n$ based on 
observations on the nodes of the box. Under the signal hypothesis, 
one observes independent Gaussian variables of unit variance at all nodes,
with zero mean off the anomalous path and mean $\mu_n$ on it. 
Under the 
\corS{null} hypothesis, one observes i.i.d. standard Gaussians on all
nodes. Arias--Castro et als. (2008) showed that if 
\corS{ the unknown directed path under the signal hypothesis has known initial location}, then detection
is possible \corS{(in the minimax sense)} if $\mu_n\gg1/\sqrt{\log n}$, while it is not possible if
$\mu_n\ll 1/\log n \sqrt{\log \log n}$. In this paper, we show that this result 
continues to hold
\corS{even when the initial location of the unknown path is not known.} 
\corOO{As
 is the case with  Arias--Castro et als. (2008)}, \corS{the 
 upper bound 
 here also applies when the path is undirected.} \corOO{The improvement 
 is achieved} 
by replacing the linear detection statistic used in Arias--Castro et als. 
(2008) with
a 
\corS{polynomial statistic, which is obtained by employing a multi-scale analysis on a quadratic statistic to bootstrap
its performance}. Our analysis is motivated by ideas developed in the context 
of the analysis of random polymers in Lacoin (2010).
}
\end{abstract}

\begin{keyword}[class=AMS]
\kwd[Primary ]{62C20, 62G10}\kwd[; secondary ]{82B20, 60K37}
\end{keyword}

\begin{keyword}
\kwd{Detecting a chain of nodes in a network} \kwd{minimax detection} 
\kwd{random scenery}
\end{keyword}
\end{frontmatter}

\section{Introduction}\label{sec:int}
\subsection{General problem}
In this paper, we will address the problem of detecting anomalous 
\corS{paths within a finite two dimensional lattice, with unknown starting point}. 
\corO{We begin
with describing the context for our results. Our presentation and motivation
are strongly influenced by \cite{ACHZ08}, to which we refer for additional
background.}

\corO{Suppose we are given a graph $G$ with node  
set $V$ and a random variable $X_v$ attached to each node 
$v\in V$.  We observe a realization of this process
and wish to know whether all the variables at the nodes have the same
behavior in the sense that they are all sampled independently
from a common distribution
$F_0$ (the null hypothesis, 
which in this paper will always be the standard Gaussian distribution), 
or whether there is a path in the network, that is, a chain of 
consecutive
nodes connected by edges, along which the variables at the nodes, still independent of each other and of the variables off the path, have a
different distribution $F_1$ (the signal hypothesis,
which in this paper
will always be the Gaussian distribution with nonzero mean and unit variance).
This is thus a composite-hypothesis testing problem.
}

\corO{In this paper, as in \cite{ACHZ08}, we focus on the case where $G$
is a box $V_n$ of side $n$ in the two dimensional Euclidean lattice, and
the path under the signal hypothesis is a \textit{directed} path. 
What
distinguishes our analysis from the case treated in \cite{ACHZ08} is that
we allow for an unknown starting point. Our main result (see
Theorem
\ref{UnknownInitial} below) is that, 
\corS{similar to} the case treated in
\cite{ACHZ08}, where the
starting point is unknown, if the mean $\mu_n$ along
the unknown path satisfies $\mu_n\corS{\ge C}/\sqrt{\log n}$ \corS{for some large constant $C$}, then detection is
possible in the sense that a sequence of tests which are asymptotically 
powerful exists. (It follows from
the main result in \cite{ACHZ08} that if $\mu_n\ll 1/\log n \sqrt{\log\log n}$,
all tests are asymptotically powerless.) It is not hard to verify that
our results concerning asymptotically powerful tests apply verbatim to the case
of \corS{undirected {\it box crossing} paths, whose starting and ending points lie on two opposite sides of $V_n$, and } to the case of 
undirected \corS{{\it annulus crossing} paths, whose starting point lies }within a
macroscopic 
sub-box of  $V_n$.
}

\corO{The main difference between the analysis here and in \cite{ACHZ08} 
is in the test \corS{which is used for the hypothesis testing}. In \cite{ACHZ08}, one uses a test 
which is \corS{based on} a linear statistics of the observations, where the weights are proportional to the inverse of the distance from the (known) initial point
of the path. These tests clearly 
cannot be used in the case where the initial point
is not known. Instead, in this paper we use tests \corS{which are based on quadratic and higher order polynomials} of the observations,
with non-homogeneous weights. These are motivated by the success that \corS{certain} quadratic forms had in the evaluation of the free energy of directed polymers
in dimension $1+1$, see \cite{La10}. We note that a naive application of 
these quadratic test \corS{statistics} leads 
to a detection threshold of order $1/(\log n)^{1/4}$ \corS{(see \textsection \ref{sec:weak bd})}. The
test we eventually use is \corS{based on} a bootstrapped version of the simple quadratic
form of \corS{the observations}, whose analysis requires us to perform a multi-scale analysis
of somewhat modified detection problems\footnote{As pointed out by 
  H. Lacoin, a similar in spirit analysis was used earlier 
  in the random polymers
  context, see \cite{BL15}.}.
}

%

\subsection{Mathematical formulation of the detection problem}  \label{formulation}
 In this section we will formalize the detection problem. We 
 consider the two dimensional lattice  $\dL^2=(\sV^2,\sE^2) $ with 
 \begin{align} 
   &\text{node set } \sV^2:=\{\vx \in \dZ^2:  x_1-x_2 \text{ is even} \},\; 
   \text{and} \notag\\
   &\text{ edge set } \sE^2:=\{\la\vx,\vy\ra: \vx, \vy \in \sV^2 \text{ and } \vx\sim\vy\},
 \label{simdef}
 \end{align}
 \corO{where $\vx\sim\vy$
 if  $x_1\ne y_1, x_2\ne y_2$ and $||\vx-\vy||_1=2$.} 
 We will use $\sH_i:=\{\vx\in\sV^2: x_1=i\}$ to denote the $i$-th hyperplane. 
 We will consider
 a set of finite two dimensional graphs  $\cG_n$, which consists of certain 
 subgraphs  of $\dL^2$ induced by \corO{nodes in} 
 the hyperplanes $\cup_{0\le i<n}\sH_i$. For $a \ge 0$ let
 $\sV_n^{(a)}$ denote the node set
 \[ \sV_n^{(a)} :=\cup_{ i=0}^{n-1} \left(\sH_i \times [-i-an, i+an]\right),  \quad \text{ and }  \quad 
  \sG_n^{(a)} := \left(\sV_n^{(a)}, \sE^2\left|_{\sV_n^{(a)}}\right.\right)\]
be the subgraph of $\dL^2$ induced by \corO{nodes in} 
$\sV_n^{(a)}$. 
 \corO{We also introduce the notation}
 \[ \sZ_n^{(a)} := \{\vz\in\dZ^2: 0\le z_1<n, |z_2| \le an+z_1\}, [n]:=\{1, 2, \ldots, n\}, [n]_0 := \{0, 1, \ldots, n-1\} \]
 and
 \[ \cG_n := \left\{\sG_n^{(a)}: a\ge 0\right\}.\] 
   
Having defined the family of two dimensional finite graphs, we 
define the collection of semi directed  nearest-neighbor paths 
(left to right crossing) on these graphs, \corO{with starting point 
on the hyperplane $\sH_0$ and endpoint on
$\sH_{n-1}$}. For \corO{ a graph}
$\sG_n=(\sV_n,\sE_n)\in\cG_n$, let 
\beq \label{P_n def} 
\sP(\sG_n) := \{\vpi=\la\vpi_0,  \ldots, \vpi_{n-1}\ra: \vpi_i \in \sH_i \text{ for all } i\in [n]_0  \text{ and } \vpi_i\sim\vpi_{i-1}\text{ for all }  i \in [n-1]\}. 
\eeq
See Figure \ref{fig: Unknown Lattice} and \ref{fig: Known Lattice} for 
an instance of such a path on graphs in \corO{$\sG_n^{(a)}$, $a=0$ and $a>0$}.
\corO{In particular, the starting point of the collection of paths
 $\sP(\sG_n)$ is known if $\sG_n=\sG_n^{(0)}$ and
unknown if $\sG_n=\sG_n^{(a)}$ with $a>0$.}

\corO{We next introduce the statistical hypothesis problem on
$(\sG_n, \sP(\sG_n))$.} \corO{To each
  node $\vv$ of the graph $\sG_n$, one  attaches
a random variable $X_\vv$.}
We assume that all random variables are independent and  consider the following hypothesis testing problem.
 \begin{itemize}
   \item {\bf Null hypothesis $H_0$:} \corO{The random variables}
     $\{X_\vv: \vv\in\sV_n\}$ are \iid with common distribution $N(0,1)$.
 \item  {\bf Alternate \corO{(signal)}
   hypothesis $H_{1,n}$:} it is a composite hypothesis
   $\cup_{\vpi \in \sP(\sG_n)} H_{1,\vpi}$, 
   where, \corO{under $H_{1,\vpi}$,
   the random variables $\{X_{\vv}: \vv\in \sV_n\}$ are independent with}
 \[ 
 X_\vv \eqd \begin{cases} N(\mu_n, 1)  & \text{  if } \vv \in \vpi \\ N(0, 1)  & \text{ otherwise}\end{cases}\quad  \text{ for some } \mu_n>0. \]
 \end{itemize}
 In other words, the null hypothesis is that the random variables $\{X_\vv\}$ represent a random scenery,  whereas the alternative hypothesis suggests that there is an anomalous path along which the mean of the random variables are  nontrivial. We refer to the above hypothesis testing problem as 
``$(\sP(\sG_n),\mu_n,\gF)$ detection problem on $\sG_n$". $\gF$ represents the cdf of $N(0,1)$.
 
\begin{figure}
    \centering
    \includegraphics[height=7cm,page=1]{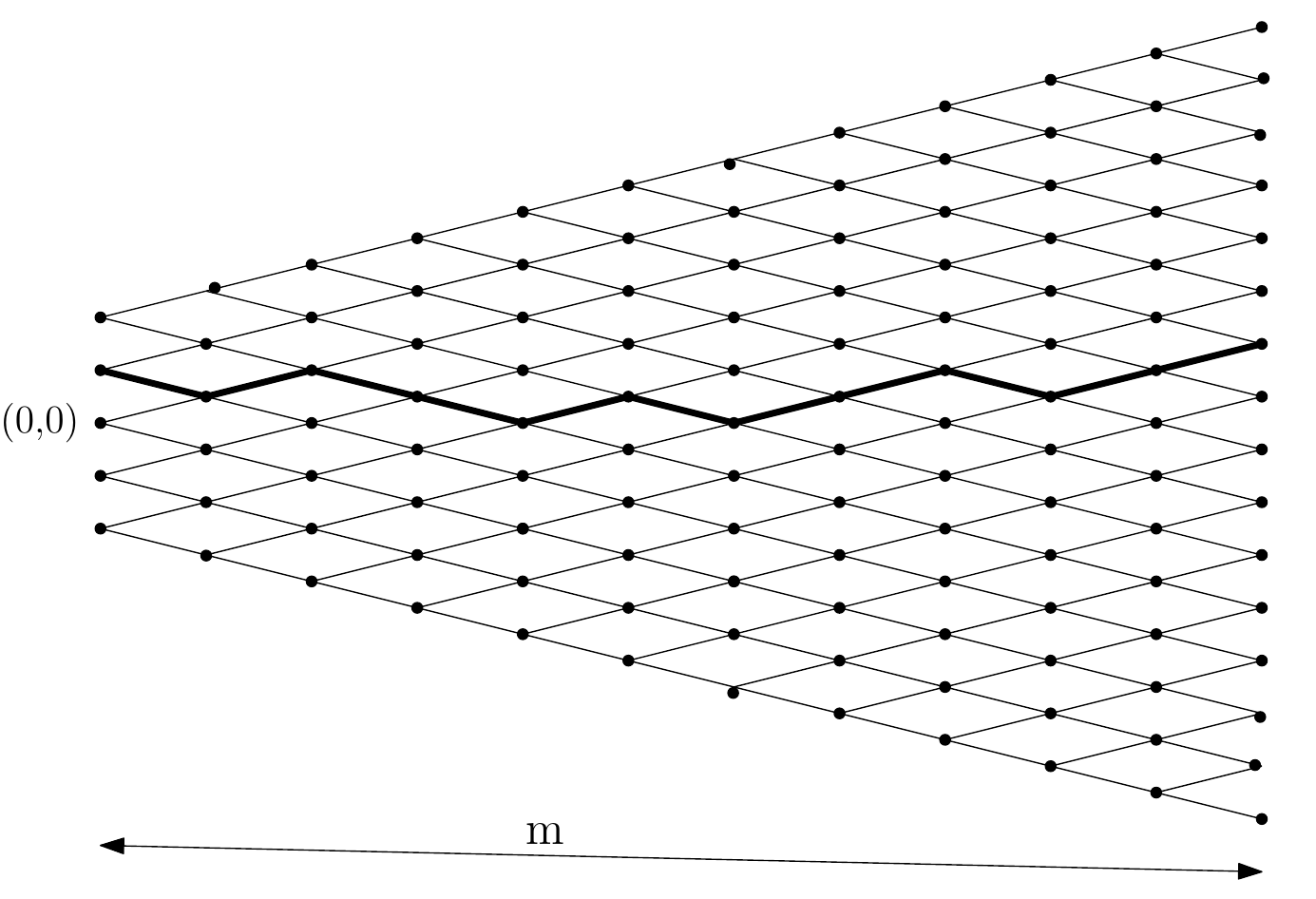}
    \caption{This is a picture of a graph in \corO{$\sG_n^{(a)}$, $a>0$.} The bold line represents a path in the two dimensional finite lattice with unknown initial location.}
    \label{fig: Unknown Lattice}
\end{figure}

\begin{figure}
    \centering
    \includegraphics[height=7cm,page=3]{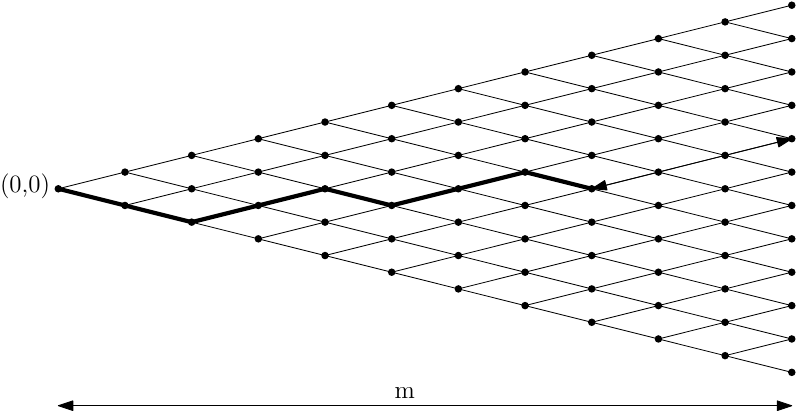}
    \caption{This is a picture of a graph in \corO{$\sG_n^{(0)}$}. The bold line represents a path in the two dimensional finite lattice with known initial location.}
    \label{fig: Known Lattice}
\end{figure}

The detection threshold is the minimum value of $\mu=\mu_n$ for which one can reliably
decide whether or not there is an anomalous path which does not follow the null distribution.
The threshold depends on the criterion used for judging the performance of the
decision rule. There are mainly two paradigms in statistical decision theory, namely
the Bayesian and the minimax approach. We will consider the second approach.
\corO{Recall} that a nonrandomized test $T_n$  is a measurable function of the collection of random variables $(X_\vv, \vv\in\sV_n)$  taking values in $\{0, 1\}$. The minimax risk of
such a test $T_n$ is defined as
\beqax
 \gc(T_n) := \pr_0(\text{Type I error}) + \sup_{\vpi\in\sP(\sG_n)} \pr_{1,\vpi}(\text{Type II error}),\;  \text{where } \\
   \pr_0(\text{Type I error}) =\pr_0(T_n=1), \quad \text{ and }  \quad \pr_{1,\vpi}(\text{Type II error})=\pr_{1,\vpi}(T_n=0).
  \eeqax
Here and later $\pr_0$ denotes the probability distribution under the null hypothesis and $\pr_{1,\vpi}$ denotes the probability distribution under the alternative hypothesis when $\vpi\in\sP(\sG_n)$ is the anomalous path.  A sequence of tests $\{T_n\}_{n\ge 1}$ for the hypothesis testing problem $(\sP(\sG_n),\mu_n)$  will be called asymptotically powerful if 
\[ \lim_{n\to \infty} \gc(T_n) =0, \]
and it will be called asymptotically powerless if
\[ \lim_{n\to\infty} \gc(T_n) \ge 1.\]

\subsection{Main result}
\corO{The main result of this paper is the following theorem.}
\begin{thm}\label{UnknownInitial}
  \corO{Fix $a\geq 0$.}
There is a finite constant $C$ large enough such that for any sequence 
of means $\{\mu_n\}_{n\ge 1}$ satisfying $\mu_n\sqrt{\log n} \ge C$, 
there exists a sequence of tests $\{T_n\}_{n\ge 1}$ for 
the hypothesis testing problem $(\corO{\sP(\sG_n^{(a)})},\mu_n,\gF)$ 
which is asymptotically powerful.  On the other hand,  for any 
sequence of means $\{\mu_n\}_{n\ge 1}$ satisfying 
\corO{$\mu_n\log n\sqrt{\log\log n }\to 0$} as $n\to\infty$, 
all sequence of tests $\{T_n\}_{n\ge 1}$ for the hypothesis 
testing problem $(\corO{\sP(\sG_n^{(a)})},\mu_n,\gF)$ 
will be asymptotically powerless.
\end{thm}
\begin{rem}
  \corO{The case $a=0$ of Theorem \ref{UnknownInitial}
  is contained in \cite{ACHZ08}.}
\end{rem}
\corS{
\begin{rem}
  The \corOO{asymptotically powerful part of the}
  assertion of Theorem  \ref{UnknownInitial} holds for the detection problem $(\tilde\sP(\sG_n^{(a)}),\mu_n,\gF)$, where $\tilde\sP(\sG_n^{(a)})$ consists of undirected paths on $\sG_n^{(a)}$ having their one endpoint in $\sH_0$ and the other endpoint in $\sH_{n-1}$. 
\end{rem}
\begin{rem}
  For $0<b<a$, the assertion of Theorem  \ref{UnknownInitial} holds for the detection problem \corO{$(\sP(a,b),\mu_n,\gF)$, where $\sP(a,b)$} 
consists of \corO{directed} paths in the subgraph of $\dL^2$ induced by $[-an , an]^2\cap\sV^2$ having their one endpoint in $[-bn , bn]^2 \cap \sV^2$ and the other endpoint on the boundary of  $[-an, an]^2\cap\sV^2$. 
\end{rem}
}

\subsection{Notation}
Through out the paper we will use the following notation.
\begin{itemize}
\item For $n \in \dN$, we will use $[n]$ to denote the set $\{1, 2, \ldots, n\}$.
\item For $\vM \in\dR^{m\times n}$, we will use $\norm{\vM}_F$ to  denote the Frobenius norm
  \[ \norm{\vM}_F := \sqrt{\textrm{Trace}
  (\vM^T\vM)}  = \sqrt{\sum_{i\in[m], j\in[n]} M_{i,j}^2}.\]
\end{itemize}
\corO{Finally, for sequences $a=(a_n)$ and $b=(b_n)$ we write $a\asymp b$ to mean that
there exists a finite universal constant $c>0$ so that $1/c\leq a_n/b_n\leq c$
for all $n$ large.}

\section{Proof of Theorem \ref{UnknownInitial}}
\subsection{Upper bound for the detection threshold}
In this section, 
we will \corS{eventually} show that \corO{a sequence of asymptotically powerful
  tests exists for \corS{the hypothesis testing problem} $H_0$ \corS{versus} $H_1$
  if $\mu_n\sqrt{\log n}>C$ for some large enough constant $C$.} \corS{We present the proof of a weaker version of this assertion in Proposition \ref{weaker version},  which is then bootstrapped in conjunction with a renormalization argument to complete the proof (see Section \ref{full proof}). First we need to introduce certain quadratic forms, which play a crucial role in the proof. }
 
  \subsubsection{Quadratic forms associated with the detection problem}
We \corO{next introduce some useful notation.} 
In order to arrange the vertices of $\sV_n$ and associated random variables in an order, we define the following partial order.
 \[ \text{For } \vx, \vy \in \sV^2, \text{ define } \vx \preccurlyeq \vy \text{ if either $x_1<y_1$ or $x_1=y_1$ and $x_2<y_2$}.\]
Using this 
partial order, we order the random variables $\{X_\vv: \vv \in \sV_n\}$ accordingly to have the $|\sV_n| \times 1$ column vector $\vX_n$.  
For $A\subset \sV_n$, we use $\ind_A$ to denote the  $|\sV_n|\times 1$ column vectors defined by
\[ \ind_A(\vv) = \begin{cases} 1 & \text{ if } \vv \in A \\ 0 & \text{ otherwise} \end{cases}  \text{ for } \vv \in \sV_n,\]
In order to describe the test $T_n$ that will separate $H_0$ and $H_1$, we also need the following equivalence relation.
\[ \text{For } \vx, \vy \in \sV^2, \text{ define } \vx \lrs  \vy \text{ if } x_1 \ne y_1 \text{ and } |x_2-y_2| \le |x_1-y_1|.\]
It is easy to see that the above is an equivalence relation. For this equivalence relation and partial order described above, we write
\[ \llb\vx\rrb := \{\vy\in\sV_n: \vx \lrs  \vy\}, \text{ and } \vx\precsim\vy\text{ if $ \vx\lrs  \vy$ and $\vx\preccurlyeq\vy$} .\]
Using the above partial order and equivalence relation, we define 
 the $|\sV_n| \times |\sV_n|$ matrix  $[\vA(\sV_n)]$ associated with the 
 \corO{full} vertex set $\sV_n$ by
\beqa \label{A:def} 
&& [\vA(\sV_n)] = \left([A(\sV_n)]_{\vx,\vy}\right)_{\vx,\vy \in \sV_n}, \text{ where } [A(\sV_n)]_{\vx,\vy}=\frac{1}{|x_1-y_1|} \mathbf 1_{\{\vx \lrs  \vy\}}, \\
&& [\bar\vA(\sV_n)] := \left(\sqrt 2\norm{[\vA(\sV_n)]}_F\right)^{-1} [\vA(\sV_n)]. \notag
\eeqa
The matrix \corO{$\vA(\sV_n)$}
will play a special role in our argument. \corO{The following lemma,
  whose proof is postponed to \textsection  
\ref{sec-quadlemproofs}, collects some of 
its elementary properties.}
\begin{lem} \label{A:properties}
For any graph $\sG_n=(\sV_n,\sE_n)\in\cG_n$ and for the matrix $[\vA(\cdot)]$ as defined in \eqref{A:def},
\beqax
 (1) & \norm{[\vA(\sV_n)]}_F  & \asymp n\sqrt{\log n}, \\
 (2) & \norm{[\vA(\sV_n)]} & = O(n), \text{ so } \norm{[\bar\vA(\sV_n)]} = O\left(1/\sqrt{\log n}\right),  \\ 
 (3) &  \ind_{\vpi}^T [\vA(\sV_n)] \ind_{\vpi}   & \asymp n\log n \;\;\text{ for any } \pi\in\sP_n,\\
 (4) &  \ind_{\vpi}^T [\vA(\sV_n)]^2 \ind_{\vpi}   & \asymp n^2 \;\;\text{ for any } \pi\in\sP_n \\
 (5) &  \ind_{\vpi}^T [\vA(\sV_n)] \textrm{\rm Diag}(\ind_{\vpi}) [\vA(\sV_n)] \ind_{\vpi}   & \asymp n(\log n)^2 \;\;\text{ for any } \pi\in\sP_n. \\
 (6) &  \norm{[\vA(\sV_n)]^2}_F  & \asymp n^4\log n, \text{ so } \norm{[\bar\vA(\sV_n)]^2}_F \asymp (\log n)^{-1/2} .
\eeqax
\end{lem}
\corO{Note that} 
Lemma \ref{A:properties} describes properties of the matrix $[\vA(\sV_n)]$. 

\corS{
\subsubsection{A weaker version of Theorem \ref{UnknownInitial}} \label{sec:weak bd}
\begin{prop} \label{weaker version}
In the set up of Theorem \ref{UnknownInitial}, if $\mu_n(\log n)^{1/4}\to\infty$ as $n\to\infty$, there is a sequence of asymptotically powerful tests for the hypothesis testing problem $H_0$ versus $H_1$.
\end{prop}
 }
 \corS{
 \begin{proof}
 Let $\vZ_n$ be 
 a $|\sV_n| \times 1$ column vector consisting of i.i.d. $N(0,1)$ random variables. Consider the quadratic form $\vZ_n^T [\vA(\sV_n)] \vZ_n$, where  $[\vA(\sV_n)]$ is the matrix defined in \eqref{A:def}. Since $[\vA(\sV_n)]$ has zero diagonal entries, each summand of  $\vZ_n^T [\vA(\sV_n)] \vZ_n$ and $Z_i(\vZ_n^T [\vA(\sV_n)] \vZ_n), i\in[|\sV_n|],$ has mean 0, as all of them are  product of independent random variables having mean 0. So 
 \beq \label{Zprop1} \E \vZ_n^T[\vA(\sV_n)]\vZ_n=0 \quad \text{ and } \quad \E[(\vZ_n^T[\vA(\sV_n)]\vZ_n)\vZ_n]=\mathbf 0.\eeq
 Also, noting that the summands $\vZ_n^T[\vA(\sV_n)]\vZ_n$  are uncorrelated, and using (1) of Lemma \ref{A:properties},
\beq \label{Zprop2}
\E[(\vZ_n^T[\vA(\sV_n)]\vZ_n)^2]= 4\sum_{(i,u)  \in \sV_n} \;\;\sum_{(j,v)\in \sV_n: (i,u) \precsim (j,v)}    (j-i)^{-2}  
 =  2\norm{[\vA(\sV_n)]}_F^2 \asymp n^2\log n.  \eeq
 Now, using the partial order $\preccurlyeq$, we order the random variables $\{X_\vv: \vv \in \sV_n\}$ attached to the nodes to have the $|\sV_n| \times 1$ column vector $\vX_n$. Define the quadratic form $Q_n:=\vX_n^T [\vA(\sV_n)] \vX_n$ and the test $T_n:=\mathbf 1_{\{Q_n>\mu_n^2 n\log n/2\}}$. In order to compute $\gc(T_n)$ note that 
\beqax 
 \vX_n &\eqd& \begin{cases} \vZ_n  & \text{ under $H_0$}\\ \vZ_n +\mu_n\mathds 1_{\vpi} & \text{ under $H_1$} \end{cases}, \text{  so using \eqref{Zprop1}, \eqref{Zprop2} and Lemma \ref{A:properties} we get}\\
 \E_0 Q_n &=&\E \vZ_n^T[\vA(\sV_n)]\vZ_n=0, \\
 \E_{1,n} Q_n &=& \E \vZ_n^T[\vA(\sV_n)]\vZ_n+ 2\mu_n\E\mathds 1_{\vpi}^T[\vA(\sV_n)]\vZ_n+\mu_n^2 \mathds 1_{\vpi}^T [\vA(\sV_n)]\mathds 1_{\vpi} \\
 &=&\mu_n^2 \mathds 1_{\vpi}^T [\vA(\sV_n)]\mathds 1_{\vpi} \asymp \mu_n^2 n\log n,\\
  \var_0 Q_n &=&\E [(vZ_n^T[\vA(\sV_n)]\vZ_n)^2]\asymp n^2\log n, \\
 \var_{1,n} Q_n &=& \var(\vZ_n^T[\vA(\sV_n)]\vZ_n+ 2\mu_n\mathds 1_{\vpi}^T[\vA(\sV_n)]\vZ_n) = \var(\vZ_n^T[\vA(\sV_n)]\vZ_n)  \\
 && + \var(2\mu_n\mathds 1_{\vpi}^T[\vA(\sV_n)]\vZ_n) + 2\cov(\vZ_n^T[\vA(\sV_n)]\vZ_n, 2\mu_n\mathds 1_{\vpi}^T[\vA(\sV_n)]\vZ_n)  \\
 &=& \E[(\vZ_n^T[\vA(\sV_n)]\vZ_n) ^2] + 4\mu_n^2 \mathds 1_{\vpi}^T[\vA(\sV_n)]\E(\vZ_n \vZ_n^T)[\vA(\sV_n)]\mathds 1_{\vpi}  \\
 && + 4\mu_n \mathds 1_{\vpi}^T[\vA(\sV_n)]\E(\vZ_n \vZ_n^T[\vA(\sV_n)]\vZ_n)
 =  \E[(\vZ_n^T[\vA(\sV_n)]\vZ_n) ^2] + 4\mu_n^2 \mathds 1_{\vpi}^T[\vA(\sV_n)]^2\mathds 1_{\vpi}  \\
 &\asymp& n^2\log n + \mu_n^2  n^2.
 \eeqax
 Using the above estimates and Chebychev inequality,
 \beqax
 \pr_0(T_n=1)
 &\le & \pr_0\left(|Q_n-\E_0Q_n|>\frac 12 \mu_n^2n\log n\right) \\
 &\le &  4\frac{\var_0(Q_n)}{(\mu_n^2 n\log n)^2} 
 \le c \frac{n^2\log n}{(\mu_n^2 n\log n)^2} =\frac{c}{\mu_n^4\log n}, \text{ and } \\
\pr_{1,n}(T_n=0)
 &\le & \pr_{1,n}\left(|Q_n-\E_{1,n}Q_n|\ge \frac 12 \mu_n^2n\log n\right) \\
 &\le & 4\frac{\var_{1,n}(Q_n)}{(\mu_n^2 n\log n)^2} 
 \le c \frac{n^2\log n+\mu_n^2 n^2}{(\mu_n^2 n\log n)^2} =\frac{c}{\mu_n^4\log n}+\frac{c}{\mu_n^2\log^2 n}  
 \eeqax
 for some constant $c$. Since the upper bounds in the above display are $o(1)$, we see that $\lim_{n\to\infty}\gc(T_n)=0$. This competes the proof.
  \end{proof}
 }
 \corS{
 Proposition \ref{weaker version} gives a weak upper bound for the detectability threshold.  In order to improve this bound,   we will use a renormalization argument. In order to employ our renormalization argument, we need to generalize the detection problem described in the introduction.  We define the necessary generalization step by step in the following section.  
 }

\subsubsection{Generalized detection problem}
Recall from Section \ref{formulation}  that $\sH_i=i+2\dZ$ denotes the $i$-th hyperplane of $\dL^2$. We extend the notion of a hyperplane by including unordered pairs of consecutive nodes from the same hyperplane. For a graph $\sG_n=(\sV_n,\sE_n)$,  define the associated {\it generalized hyperplanes} as follows.
\[ \wt\sH_i(\sG_n) := \left\{\{\vv\}: \vv\in\sV_n \text{ and }  v_1=i\right\} \cup \left\{\{\vu,\vv\}: \vu, \vv\in\sV_n,  v_1=u_1=i \text{ and } |v_2-u_2|=2\right\}, i\in[n]_0.\]
So each $\wt\sH_i(\sG_n)$ consists of singletons and doubletons.  We extend the definition of neighboring relationship $\sim$ defined in \eqref{simdef} to a new relation $\ws$ on $\cup_{i\in[n]_0} \wt\sH_i(\sG_n)$.  We say that $A\in\wt\sH_i(\sG_n)$ and $B \in \wt\sH_j(\sG_n)$ are neighbors \ie
\[  A\ws B, \text{ if $ a\sim b$ for some  $a\in A$ and $b\in B$. }\]   
In the same spirit, a {\it generalized path} $\vPi$ on $\sG_n$  will be union of finite sequences of successive neighbors from $\cup_{i\in[n]_0} \wt\sH_i(\sG_n)$. A generalized path may be  incomplete in the sense that it may not intersect all hyperplanes. Define
\beqax 
\wt\sP(\sG_n) :=\left\{\vPi=\la\vPi_0,  \ldots, \vPi_{n-1}\ra: \right.&& \text{ for all } i\in[n]_0 \text{ either } \vPi_i=\emptyset \text{ or }  \vPi_i\in \wt\sH_i(\sG_n),   \\
&& \left.\text{ and } \vPi_i\ws\vPi_{i-1} \text{ whenever } \vPi_i, \vPi_{i-1} \ne\emptyset \right\}. 
 \eeqax
 See Figure \ref{generalized path} for a picture of such a generalized path for a graph in $\cG_n$.
 \begin{figure}
    \centering
    \includegraphics[height=5.5cm,page=2]{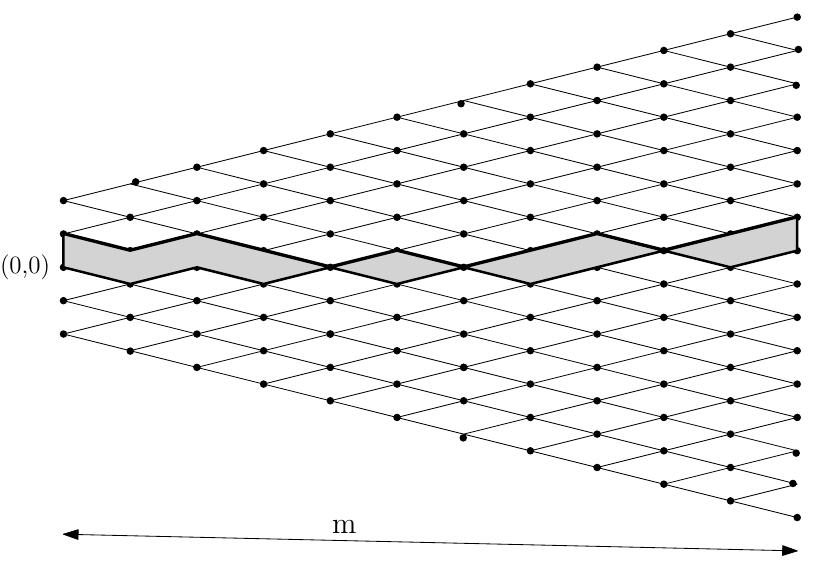}
  \hfill \includegraphics[height=5.5cm,page=10]{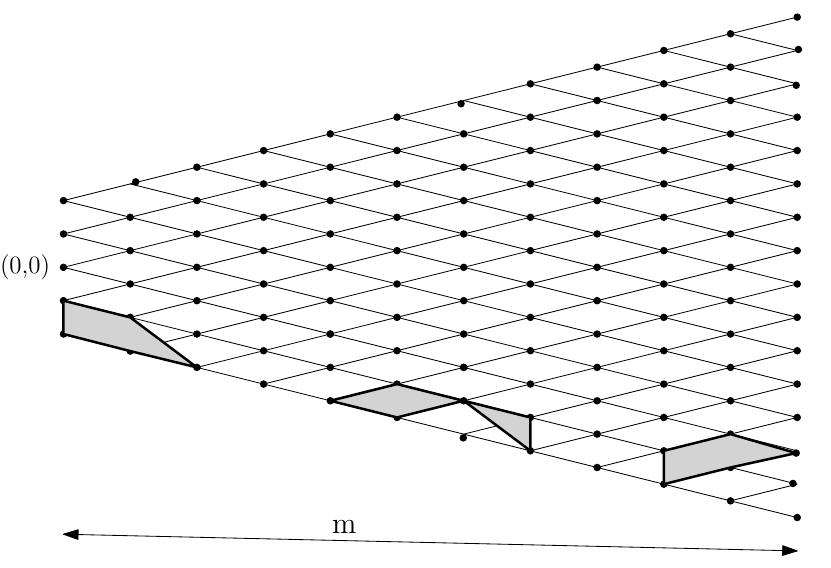}
     \caption{This is a picture of a graph in $\cGU_n$. The shaded region represents a generalized path in the two dimensional finite lattice with unknown initial location.}
    \label{generalized path}
\end{figure}
 In the generalized detection problem, we also assume that each node $\vv$ of the graph $\sG_n$  has a (observable) random variable $X_\vv$ associated with it, and the random variables $\{X_\vv\}$ are independent.   We will refer to the collection of random variables $(X_\vv, \vv\in\sV_n)$ as {\it observables}.   Suppose  $(\nu_{\vv,\vPi} \in\dR_+, \vv\in\sV_n, \vPi\in\wt\sP(\sG_n))$ is a collection of {\it signals} and $(Z_{\vv}, \vv\in\sV_n)$ and $(Y_{\vv,\vPi}, \vv\in\sV_n, \vPi\in\wt\sP(\sG_n))$, which we will refer to as {\it basic noise} and {\it additional noise} respectively, are two collections of (possibly unobservable) random variables satisfying the following properties.
  \begin{pro} \label{NoiseAssump}
  The noise variables satisfy the following.
  \begin{enumerate}
   \item The random variables $((Z_{\vv}, Y_{\vv,\vPi}), \vv\in\sV_n)$ are independent. 
 \item $Y_{\vv,\vPi}$ has mean zero and variance at most 1. 
\item For each $\vv\in\sV_n$ the random variables $Z_{\vv}$ and $Y_{\vv,\vPi}$ are uncorrelated, although they can be dependent.
 \end{enumerate}
 \end{pro}
Based on the signals, basic noise and additional noise variables as described above,  we consider the following two hypotheses.
 \begin{itemize}
 \item {\bf Null hypothesis $H_0$:} $X_\vv=Z_{\vv}$ for all $\vv\in\sV_n$.
 \item  {\bf Alternate \corO{(signal)}
   hypothesis $H_{1,n}$:} 
   \corO{it is a composite hypothesis}
   $\cup_{\vpi \in \wt\sP(\sG_n)} H_{1,\vpi}$, 
   where, \corO{under $H_{1,\vpi}$},
 \beq \label{H1Pudef}
   X_\vv  = \begin{cases} Z_{\vv}+Y_{\vv,\vPi}+\nu_{\vv,\vPi}  & \text{  if } \vv \in \cup_{i=0}^{n-1}\vPi_i  \\  Z_{\vv}  & \text{ otherwise}\end{cases}.  \eeq 
   \end{itemize}
 We refer to this hypothesis testing problem as ``generalized detection problem".
 
 The following proposition  summarizes the relevant properties of a certain quadratic statistic under the hypotheses of a  generalized detection problem. Proposition \ref{Quad} will play an important role in proving the upper bound in Theorem \ref{UnknownInitial}.

\begin{prop} \label{Quad}
  \corO{ Fix $a\geq 0$.}
  Let $\sG_{\corO{n}}^{\corO{(a)}}=
  (\sV_{\corO{n}}, \sE_{\corO{n}})$ be any graph in 
  $\cG_{\corO{n}}$,  $\{F_\vv\}_{\vv\in\sV_{\corO{n}}}$ 
  be any collection of distributions on $\dR$ having mean 0, 
  variance 1 and finite third moment. 
  Consider the generalized hypothesis testing problem  
  with observables $(X_\vv, \vv\in\sV_{\corO{n}})$, where the   basic noise variables $(Z_{\vv}, \vv\in\sV_{\corO{n}})$, 
  the additional noise variables $(Y_{\vv,\vPi}, \vv\in\sV_{\corO{n}}, 
  \vPi\in\wt\sP(\sG_{\corO{n}}))$ and the signals $(\nu_{\vv,\vPi}, 
  \vv\in\sV_{\corO{n}}, \vPi\in\wt\sP(\sG_{\corO{n}}))$ 
  satisfy Property \ref{NoiseAssump} and $Z_\vv$ has distribution $F_\vv$.
  Let $Q_{\corO{n}}:=Q_{\corO{n}}[[\bar\vA(\sV_{\corO{n}})], 
(X_\vv, \vv\in\sV_{\corO{n}})]$ and $W_{\corO{n}}:=W_{\corO{n}}[[
\bar\vA(\sV_{\corO{n}})], (Z_{\vv}, \vv\in\sV_{\corO{n}})]$ be the 
quadratic forms based on the observables and basic noise variables respectively, where $[\bar\vA(\cdot)]$ is as in \eqref{A:def}.  
 \begin{enumerate}
   \item Let $\gb_3:=\max_{\vv\in\sV_{\corO{n}}} \E(|Z_{\vv}|^3)$. \corO{Then
   there exists a constant $C=C(a)$ so that}
 \beq \label{Wmbd}
 \sup_{x\in\dR} |\pr(W_{\corO{n}} \le x) - \gF(x)| \le C \left[ (\log 
 \corO{n})^{-1/2}+\left(\frac{\gb_3}{\corO{n}}+\frac{\gb_3^2}{\corO{n}
 \log^{3/2}\corO{n}}\right)^{1/4}\right].
 \eeq
 \item Let $\bar \nu := \max_{\vv\in\vPi} \nu_{\vv,\vPi}$ . 
   There is a random variable $U_{\corO{n}}$, 
   satisfying $\E U_{\corO{n}}=\E U_{\corO{n}}W_{\corO{n}} =0$ and 
   $\E U_{\corO{n}}^2\le C\bar\nu^2/\log \corO{n}$, 
   such that
  \[
    Q_{\corO{n}} \eqd \begin{cases} W_{\corO{n}}  & \text{ under $H_0$} \\
   W_{\corO{n}} + U_{\corO{n}} + \nu(\sG_{\corO{n}}) & \text{ under } 
   H_{1,\vPi} \end{cases},    \text{ where } 
\nu(\sG_{\corO{n}}):=\sum_{\vv,\vv'\in\sV_{\corO{n}}} \nu_{\vv,\vPi} 
\nu_{\vv',\vPi} [\bar\vA(\sV_{\corO{n}})]_{\vv,\vv'}.
 \]
\item    For $I\subset[{\corO{n}}]$ let $\underline \nu_I := 
  \min_{i\in I}\max_{\vv\in\vPi_i} \nu_{\vv,\vPi}$. 
  There is a constant $c\corO{=c(a)}\in(0,1)$  such that
\beq \label{nuGm def}
  c\underline \nu_I^2 
\le \frac{\nu(\sG_{\corO{n}}) }{\sqrt{\log {\corO{n}}}}\le c^{-1} 
\bar \nu^2, \text{ 
  for any $I\subset\{i\in[{\corO{n}}]_0:\vPi_i\ne\emptyset\}$ satisfying 
$|I|\ge {\corO{n}}/2$.}
\eeq
\end{enumerate}
\end{prop}
In the above set up, we will interpret $W_{\corO{n}}, U_{\corO{n}}$ and
$\nu(\sG_{\corO{n}})$ as the basic noise variable, 
additional noise variable  and signal, for the graph $\sG_{\corO{n}}$ 
respectively. 

The proof of Proposition \ref{Quad} uses the properties of
the matrix $\vA(\sV_n)$ contained in Lemma \ref{A:properties}. The proof is postponed to \textsection \ref{QuadProof}. The last ingredients that we need for proving Theorem \ref{UnknownInitial} are some distributional properties of quadratic forms, which we present in the following section.

\subsubsection{Moment bounds and Gaussian approximation for quadratic forms}
Let $(X_j, j\in[n])$ denote independent random
variables  such that $\E X_j = 0$ and $\E X_j^2= 1$ for all $j\ge 1$. Let $\vA
=\{a_{j,k}\}_{j,k=1}^n \in \dR^{n\times n}$ be such that 
\beq \label{QFproperty}
\vA \text{ is a symmetric matrix}, \quad a_{j,j}=0 \text{ for all  } j\in[n],  \quad  \text{ and }  \quad \textrm{Trace}(\vA^2) =\frac 12.\eeq 
Consider the quadratic forms
\beq \label{Qndef} 
Q_n[\vA, (X_j, j\in[n])] := \sum_{j,k=1}^n a_{j,k}X_jX_k \quad \text{ and } \quad  G_n(\vA) := \sum_{j,k=1}^n a_{j,k} Y_jY_k, \eeq
 where $(Y_j, j\in[n])$  are \iid with common distribution $N(0,1)$. 
 Keeping in mind that we will need upper bounds for the third moment of certain quadratic forms, we state  the following moment estimate.
\begin{thm}[Theorem 2 of \cite{W60}] \label{Quad Moment Bd}
If $\vA\in\dR^{n\times n}$ satisfies \eqref{QFproperty} and $(X_j, j\in[n])$  are  independent random variables having zero mean, then
\[ \E(|Q_n[\vA,(X_j, j\in[n])]|^s) \le 2^{5s/2} \gC(s/2+1/2) (\gC(s+1/2))^{1/2}\max_{j\in[n]}\E(X_j^{2s}). \]
\end{thm}

Other than moments, we will also need error bounds for Gaussian approximation of quadratic forms. In this context,
Rotar' \cite{R73} proved that under sufficiently weak conditions on the matrix
$\vA$ and for large $n$, the distribution of $Q_n[\vA,(X_j, j\in[n])]$ is close to that of $G_n(\vA)$.
Gamkrelidze and Rotar' \cite{GR78} obtained bounds for the error of this approximation,
which were improved by Rotar' and Shervashidze \cite{RS85}. Here is their result. 
\begin{align}
  \text{Let }& F_j(x)=\pr(X_j\le x),  \quad \gF(x)=\int_{-\infty}^x\frac{1}{\sqrt{2\pi}} \corO{e^{-y^2/2} \; dy,} \quad \nu_j := 3\int_{-\infty}^\infty x^2|F_j(x)-\gF(x)|\; dx, \notag \\
& s_j^2 := \sum_{k=1}^n a_{n,j,k}^2 , \quad L := \sum_{j=1}^n \nu_js_j^3 + \sum_{j, k=1}^n \nu_j \nu_k |a_{n,j,k}]^3, \quad \gD := \textrm{Trace}(\vA^4).   \label{LDelta def}
\end{align}
\begin{thm}[See  \cite{RS85}] \label{IND Quad Bd}
Assume that \eqref{QFproperty} holds and $L, \gD$ are as in \eqref{LDelta def}. If $\gD<1/2$,  then  there is an absolute constant $C$ such that
\[ \sup_{x\in\dR} \left|\pr(Q_n[\vA,(X_j, j\in[n])] \le x) - \pr(G_n(\vA) \le x)\right| \le C(1-\log(1-2\gD))^{3/4}L^{1/4}.  \]
\end{thm}
Later,  G\"otze and Tikhomirov \cite{GT99} obtained improved bound for the Kolmogorov distance between normalized quadratic forms of \iid \corO{Gaussian}
random variables and the
Gaussian distribution.
\begin{thm}[Theorem 1 of \cite{GT99}] \label{IID Quad Bd}
Assume that \eqref{QFproperty} holds and let $G_n(\cdot)$ be as in \eqref{Qndef}. Then  
\[ \sup_{x\in\dR} \left|\pr(G_n(\vA) \le x) - \gF(x)\right| \le  C \norm{\vA} \quad \text{  for some absolute constant $C$}. \]
\end{thm}
We will apply Theorem \ref{IND Quad Bd} and \ref{IID Quad Bd} 
\corO{to bound the 
  Kolmogorov distance between the
quadratic form }
  $Q_n[[\bar\vA(\sV_n)],(U_\vv, \vv\in\sV_n)]$, where $[\bar\vA(\cdot)]$ is defined in \eqref{A:def},   and the
Gaussian distribution. 

\begin{prop} \label{Quad Kolmogorov Bd}
  \corO{Fix $a\geq 0$. For any graph $\sG_n:=
    \sG_n^{(a)}=(\sV_n,\sE_n)\in\cG_n$} 
  and  
any collection of random variables $(U_\vv, \vv\in\sV_n)$  
having zero mean, unit variance and finite third moment, if 
$[\bar\vA(\sV_n)]$ is the matrix as defined in \eqref{A:def}, 
$Q_n=Q_n[[\bar\vA(\sV_n)], (U_\vv, \vv\in\sV_n)]$ 
is the quadratic form as defined in \eqref{Qndef} and 
$\gb_3 :=\max_{\vv\in\sV_n} \E(|U_\vv|^3)$, then 
\[ \sup_{x\in\dR} \left|\pr(Q_n \le x) - \gF(x)\right| \le C\left[ (\log n)^{-1/2}+
\left(\frac{\gb_3}{n}+\frac{\gb_3^2}{n\log^{3/2}n}\right)^{1/4}\right]\,,\]
\corO{where $C=C(a)$ is an absolute constant.}
\end{prop}
The proof of Proposition \ref{Quad Kolmogorov Bd} uses properties of
the matrix $\vA(\sV_n)$ contained in Lemma \ref{A:properties}. The proof is postponed to \textsection \ref{Quad K Bd proof}.

\subsubsection{Proof of  the upper bound  for the detectability threshild} \label{full proof}
Proposition \ref{Quad} will \corO{play a crucial role in the
  proof  of Theorem \ref{UnknownInitial}. It will be used in conjunction with
a renormalization argument.}

\begin{proof}[Proof of Theorem \ref{UnknownInitial}, upper bound.]
  \corO{Throughout the proof, we fix $a\geq 0$ and write  $\sG_n=\sG_n^{(a)}$.}
Let $n_0:=n, n_k:=n_{k-1}^{\eps_k}$ for $k\in[K+1]$, where $K \in\dN$ and $\eps_1, \ldots, \eps_K \in(0,1)$ will be specified later \corO{(see
\eqref{eq-021617a} and \eqref{eq-021617b})}, and set $\eps_{K+1}:=0$. 

We construct a sequence of hierarchical partition $\{\cB_k\}_{0\le k\le K}$ of $\sG_n$.
$\cB_k$ will consist of vertex-disjoint subgraphs of $\sG_n$, where each subgraph is induced by certain vertices of $\sG_n$ which reside within either a square or a right angled isosceles triangle having side length $n_k$.  The index sets for these partitions are $\{\vV_k\}_{k=0}^K$, where $\vV_0=\emptyset$ and $\vV_k\subset\sZ_{n_0/n_1}^{(a)}\times \cdots \times \sZ_{n_{k-1}/n_k}^{(a)}$  for $k\in [K]$.  For $k=0$, $\cB_0$ is the singleton trivial partition $\{B^{(0)}_\emptyset:=\sV_n^{(a)}\}$. For $k=1$ let $\vV_1:=\{\vu\in\sZ_{n_0/n_1}^{(a)}: V^{(1)}_\vu \ne\emptyset\}$, where 
\[ V^{(1)}_\vu:=
\sV_{n_0}^{(a)} \cap  \corO{
[u_1n_1, u_1n_1+n_1) \times  [u_2n_1, u_2n_1+n_1).}
\]
For $\vu\in\vV_1$, define $B^{(1)}_\vu$ to be the subgraph of $\sG_n$ induced by $V^{(1)}_\vu$
and define $\cB_1:=\{B^{(1)}_\vu: \vu\in\vV_1\}$. It is clear that for each $\vu\in\vV_1$, the vertices belonging to $V^{(1)}_\vu$ reside within either a complete square  or a diagonally halved square (right angled isosceles triangle) having side length $n_1$. Also, the vertex sets $\{V^{(1)}_\vu: \vu\in\vV_1\}$ are disjoint. 
\begin{figure}
    \centering
    \includegraphics[height=6cm,page=11]{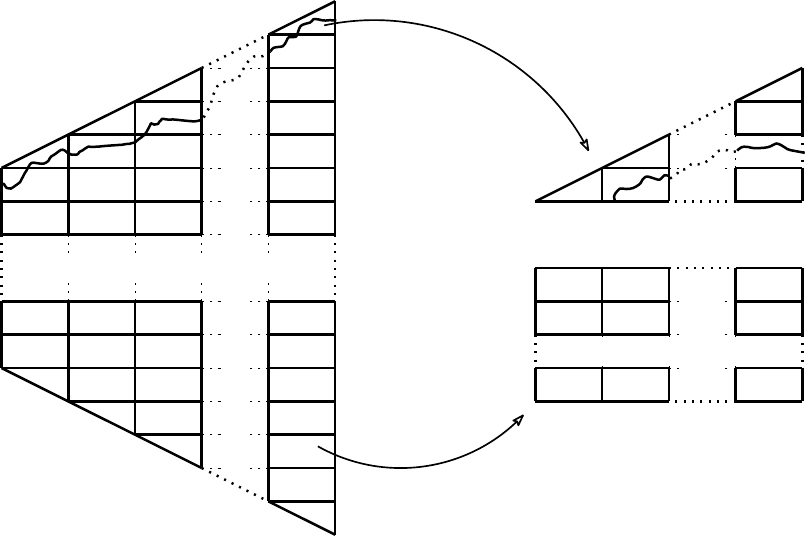}
    \hfill
     \includegraphics[height=6cm,page=12]{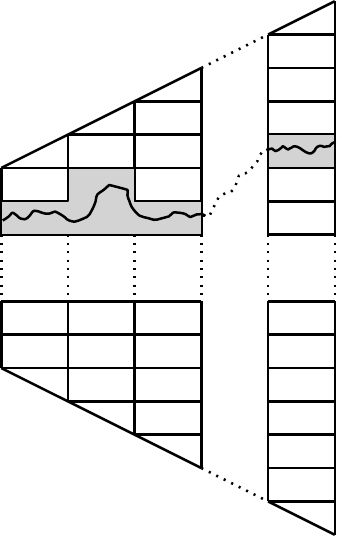}
    \caption{The left figure gives a sketch of the partitions $\cB_k, k=0, 1, 2$.  The whole graph is $B^{(0)}_\emptyset$. The larger squares correspond to subgraphs in $\cB_1$ and the small squares correspond to subgraphs corresponding to $\cB_2$. The right figure shows the corresponding generalized path on $G^{(0)}_\emptyset$.}
    \label{fig:GraphPartition}
\end{figure}
\begin{figure}
    \centering
    \includegraphics[height=7cm,page=13]{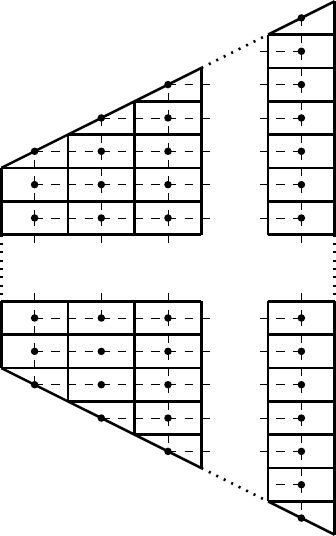}
    \caption{This is a sketch of the coarse grained graph. Here the vertex set consists of the squares and the edges are represented by the links.}
    \label{fig:coarse grained graph}
\end{figure}
Having defined $\cB_k$ for \corO{some}
$k\in [K]_0$, we obtain $\cB_{k+1}$ as follows. Note that  $B^{(k)}_{\vv} \in \cB_k$ is the subgraph of $\sG_n$ induced by the vertex set $V^{(k)}_\vv$, and the vertices belonging to $V^{(k)}_\vv$ reside within either a complete square or an isosceles triangle having side length $n_k$. So, $V^{(k)}_\vv$ can be divided into disjoint subsets $V^{(k+1)}_{\vv,\vu}, \vu\in\sZ_{n_k/n_{k+1}}^{(a)}$ as follows.  $V^{(k)}_\vv$ is a spatial translate (say $\tau$) of either $\sV_n\cap \{(x,y): 0\le x,y<n_k\}$ or $\sV_n\cap \{(x,y): 0\le y\le x<n_k\}$. 
We take $V^{(k+1)}_{\vv,\vu}$ to be the intersection of $V^{(k)}_\vv$ 
and the image under $\tau$  of 
\corO{$[u_1\frac{n_k}{n_{k+1}}, u_1\frac{n_k}{n_{k+1}}+n_{k+1}) \times
  [u_2\frac{n_k}{n_{k+1}}, u_2\frac{n_k}{n_{k+1}}+n_{k+1})$.} 
Having defined $V^{(k+1)}_{\vv,\vu}$, we define 
$\vV_{k+1}:=\{(\vv,\vu): \vv\in\vV_k,  V^{(k+1)}_{\vv,\vu}\ne\emptyset\}$. 
For  $\vv\in\vV_{k+1}$, $B^{(k+1)}_\vv$ denotes the subgraph of $\sG_n$  
induced by vertex subset $V^{(k+1)}_\vv$ and define 
$\cB_{k+1}:=\{B^{(k+1)}_\vw: \vw\in\vV_{k+1}\}$.   
See Figure \ref{fig:GraphPartition} for a sketch of these partitions. 
 
 After defining the sequence of partitions $\{\cB_k\}_{k=0}^K$ as above, we will assign a random variable $Q^{(k)}_\vv$ to the subgraph $B^{(k)}_\vv$ for all $\vv\in\vV_k$ and $0\le k\le K$.
 These random variables will be defined using backward induction in $k$. For $\vv\in\vV_K$,  we order the random variables $\{\vX_\vu: \vu \in V^{(K)}_\vv\}$ using the partial order $\preccurlyeq$ to obtain the column vector $\vX^{(K)}_\vv$ consisting of $O(n_K^2)$ many entries. Then $Q^{(K)}_\vv$ is taken to be the quadratic form
 \[  Q^{(K)}_\vv :=  (\vX^{(K)}_\vv)^T [\bar\vA(V^{(K)}_\vv)] \vX^{(K)}_\vv, \]
where $[\bar\vA(\cdot)]$ is as defined in \eqref{A:def}. Having defined the random variables $\{Q^{(l)}_\vv:  \vv \in\vV_l\}$ for all $l\in [K] \setminus [k]_0$ we obtain the random variables $\{Q^{(k-1)}_\vu: \vu\in\vV_{k-1}\}$ as follows. If $\vu\in\vV_{k-1}$, then $B^{(k-1)}_\vu$ can be thought of as a (coarse-grained) graph $\bar B^{(k-1)}_\vu$ (see figure \ref{fig:coarse grained graph}) having vertex set 
$\bar V^{(k-1)}_\vu  := \{ B^{(k)}_{\vu,\vw}: (\vu,\vw) \in \vV_k\}$ and  edge set corresponding to the neighboring relation:  
$B^{(k)}_{\vu,\vw} \ws B^{(k)}_{\vu,\vw'}$ if \corO{$\|\vw -\vw'\|_\infty=1$.}
For completeness, we define $\bar B^{(K)}$ and  $\bar V^{(K)}$ to be same as $B^{(K)}$ and $V^{(K)}$ respectively.
We order the random variables $\{Q^{(k)}_{\vu, \vw}: (\vu,\vw) \in \vV_{k}\}$ using the partial order $\preccurlyeq$ and obtain the vector $\vX^{(k-1)}_\vu$ consisting of $O((n_{k-1}/n_k)^2)$ entries. Then we define
\[  Q^{(k-1)}_\vu :=  (\vX^{(k-1)}_\vu)^T [\bar\vA(\bar V^{(k-1)}_\vu)]\vX^{(k-1)}_\vu, \]
where $[\bar\vA(\cdot)]$ is as defined in \eqref{A:def}.  
We proceed in this way until $Q^{(0)}_\emptyset$ is defined.
Thus, $Q^{(0)}_\emptyset$ is a quadratic form in terms of the random variables   $\{Q^{(1)}_\vv: \vv \in \vV_1\}$, where each $Q^{(1)}_\vv$ is again a quadratic form in terms of the random variables $\{Q^{(2)}_{\vv,\vu}: (\vv,\vu) \in \vV_2\}$, and so on.  

Our next goal is to study the distribution of $Q^{(0)}_\emptyset$ under $\pr_0$ and $\pr_{1,n}$. Choose and fix any path $\vpi \in \sP(\sG_n)$.  
Note that for each $0\le k\le K$ and $\vv\in\vV_k$,  $\vpi$ induces  generalized paths $\vPi^{(k)}_\vv\in\wt\sP(\bar B^{(k)}_\vv)$  \corO{for}
the (possibly coarse-grained) graph $\bar B^{(k)}_\vv$. See figure \ref{fig:GraphPartition} for a sketch of an anomalous path $\vpi$ and corresponding (coarse-grained) generalized paths $\vPi^{(0)}_\emptyset$ and $\vPi^{(1)}_\vv$ for some $\vv\in\vV_1$.
Using this fact and Proposition \ref{Quad} we \corO{next construct}
the basic noise variables $(W^{(k)}_\vv, \vv\in \vV_k)$, 
additional noise variables $(U^{(k)}_{\corO{\vv}}, \vv\in \vV_k)$ 
and signals $(\nu^{(k)}_{\corO{\vv}}, \vv\in \vV_k)$ for $0\le k\le K$ inductively as follows. 
\corO{(Since $\vpi$ is fixed, we eliminate it from the notation, writing
  e.g. $U^{(K)}_{\vv}$ for
  $U^{(K)}_{\vv,\vpi}$.)}

First we define the attributes at level $K$. For $\vv\in\vV_K$,  
we apply Proposition \ref{Quad}, \corO{with $n$ replaced by $n_K$},
basic noise variables $(X_\vu, \vu\in V^{(K)}_\vv)$, 
which are the basic noise variables for the vertices of $\sG_n$, 
additional noise variables given by zeros and  
signals $(\mu\mathbf 1_{\{\vu\in\vpi\}}, \vu\in V^{(K)}_\vv)$, 
and obtain the basic noise variable,  additional noise variable and signal, 
which we will denote by 
$W^{(K)}_\vv, U^{(K)}_{\corO{\vv}}$ and $\nu^{(K)}_{\corO{\vv}}$ respectively.

  Having obtained $((W^{(k)}_\vv, U^{(k)}_{\corO{\vv}}, \nu^{(k)}_{\corO{\vv}}), \vv\in\vV_k)$ for $k\in[K]$, we obtain the attributes  
  $(W^{(k-1)}_\vv, U^{(k-1)}_{\corO{\vv}}, \nu^{(k-1)}_{\corO{\vv}})$ 
  by applying
 Proposition \ref{Quad} with basic noise variables 
 $(W^{(k)}_{(\vv,\vu)}, (\vv,\vu)\in \bar V^{(k-1)}_\vv)$, additional 
 noise variables  $(U^{(k)}_{\corO{(\vv,\vu)}}, 
 (\vv,\vu)\in \bar V^{(k-1)}_\vv)$ and  signals 
 $(\nu^{(k)}_{\corO{(\vv,\vu)}}, (\vv,\vu)\in \bar V^{(k-1)}_\vv)$. 
 We proceed in this way until $W^{(0)}_\emptyset, U^{(0)}_{\corO{\emptyset}}$ 
 and $\nu^{(0)}_{\corO{\emptyset}}$ are defined. 

 \corO{We next}
 estimate the signals $(\nu^{(k)}_{\corO{\vv}}, k\in[K],\vv\in\vV_k)$. 
 \corO{To do so,} we first 
 define certain vertical segments $(S^{(k)}_\vi, 
 \vi\in\vI_k:=\otimes_1^k [n_{i-1}/n_i]_0,  k\in[K+1])$ of $\sG_n$, which 
 we  call {\it slabs}. Define  
\[ S^{(0)}_\emptyset := \sV_n, \text{ and } 
S^{(k)}_\vi  := \left\{\vv\in\sV_n: \sum_{j=1}^k i_j n_j \le v_1<\sum_{j=1}^k i_j n_j+n_k\right\}, \text{ for }  \vi \in \vI_k \text{ and } k\in[K+1].
 \]
 \corO{Note that the slabs are line segments if $k=K+1$.}
 We also define the projection map 
 \[\fp:\cup_{k\in[K]} \vV_k \mapsto \corO{\cup}_{k\in[K]}\vI_k \text{ by assigning } \fp(\vv) := \vi, \text{ if } \vv_j \text{ has first component } i_j \text{ for all }  j\in[k].\]  
We call a subgraph $B^{(k)}_\vv$ {\it touched} if $\vpi$ intersects it.  We call a subgraph $B^{(K)}_\vv$ {\it good} if $\vpi$ intersects 
it in at least $\frac 12 n_{K}$ slabs \corO{(=line segments)}, \ie
among $S^{(K+1)}_{\corO{(}\fp(\vv),i_{K+1}i\corO{)}}, i_{K+1}\in[n_{K}]_0$.
\[ B^{(K)}_\vv \text{ is  `good' if } \left|\left\{i_{K+1}\in[n_{K}]_0: \pi_i\in B^{(K)}_\vv \text{ for } i=\sum_{j=1}^K v_{j,1}n_j+i_{K+1}\right\}\right|\ge \frac 12 n_{K}. \] 
We extend the definition of good subgraphs to other 
levels inductively as follows. For $k\in[K-1]$, 
we call a subgraph $B^{(k)}_\vv$ {\it good} if there are at least 
$\frac 12 (n_k/n_{k+1})$ many slabs among $S^{(k+1)}_{\corO{(}
\fp(\vv),i_{k+1}\corO{)}}, 
i_{k+1}\in[n_k/n_{k+1}]_1$ where $B^{(k)}_\vv$ contains at least one good subgraph $B^{(k+1)}_{\corO{(}\vv,\vu\corO{)}}$,  \ie
\[ B^{(k)}_\vv \text{ is `good' if  } \left|\left\{i_{k+1} \in[n_k/n_{k+1}]_0:  \exists\text{  good  $B^{(k+1)}_{\corO{(}\vv,\vu\corO{)}}$
satisfying $\fp(\vu)=i_{k+1}$}\right\}\right| \ge \frac{n_k}{2n_{k+1}}.\]

\corO{The following lemma contains the required control
  on the signal variables $\nu^{(k)}_{\vv}$.}
\corO{
  \begin{lem}
    \label{lem-CD}
There is a constant 
$c\in(0,1)$ such that the following holds  for all $k\in[K+1]_0$.
\begin{align}
 (A) \;&\; \nu^{(k)}_{\vv} \ge  \underline \nu^{(k)} := \frac 1c 
 \exp\left[2^{K-k+1}\log(c\mu)\right] \prod_{l=k}^{K} 
 [\log(n_l/n_{l+1})]^{2^{l-k-1}}  \text{ if $B^{(k)}_\vv$ is good}. 
 \notag\\
 \;&\;\label{ind hyp} \\
 (B) \;&\; \nu^{(k)}_{\vv} \le  \bar \nu^{(k)} := c \exp\left[2^{K-k+1}\log(\mu/c)\right] \prod_{l=k}^{K} [\log(n_l/n_{l+1})]^{2^{l-k-1}}  \text{ if $\vpi$ intersects $B^{(k)}_\vv$}.  \notag
 \end{align}
 \end{lem}
 }
 \corO{
   \begin{rem}
     Note that the condition that $\vpi$ intersects $B^{(k)}_\vv$
 in (B) is made only for aestetic reasons: the claim is obvious otherwise,
 for then $\nu^{(k)}_\vv=0$.
 \end{rem}}
 \begin{proof}[Proof of Lemma \ref{lem-CD}]
   \corO{We begin with the following facts.}
\corO{\begin{align*}
(I) \;&\; \text{ Each of the slabs } S^{(k)}_\vi, \vi\in \vI_k, \text{ has at least one good $B^{(k)}_\vv$ satisfying $\fp(\vv)=\vi$}. \notag \\
(II) \;&\; \text{ Each of the slabs } S^{(k)}_\vi, \vi\in \vI_k, \text{ has at most two  touched $B^{(k)}_\vv$ satisfying $\fp(\vv)=\vi$}. \notag 
 \end{align*}
 }

 \corO{To see (II), note that}
  since the slabs $S^{(k)}_\vi, \vi\in\vI_k,$ have width $n_k$,  
the subgraphs $\{B^{(k)}_\vv: \fp(\vv)=\vi\}$ constitute 
a partition of $S^{(k)}_\vi$, and each $B^{(k)}_\vv$ \corO{resides} 
within a square (or isosceles triangle at the boundary of $S^{(k)}_\vi$) 
having side length $n_k$, 
it \corO{follows}
that for each $k\in[K]$, the 
path $\vpi$ intersects each $S^{(k)}_\vi$ in either 
one subgraph $B^{(k)}_\vv\in\cB_k$ satisfying $\fp(\vv)=\vi$ 
or in two (consecutive   and disjoint) subgraphs 
$B^{(k)}_\vv, B^{(k)}_\vu\in\cB_k$ satisfying 
$\fp(\vv)=\fp(\vu)=\vi$. This proves (II).  

 \corO{To see (I), we use  induction on $k$. We first
 show the induction basis.}
Since each slab $S^{(K)}_\vi$ has $n_K$ 
hyperplanes and \corO{$\vp$ crosses each hyperplane}, 
for each slab $S^{(K)}_\vi$, 
there is at least one good subgraph  
$B^{(K)}_\vv$ satisfying $\fp(\vv)=\vi$, \corO{showing (I) for $k=K$}.


Now suppose \corO{(I)}
hold for  $k=l+1$.  
So each slab $S^{(l+1)}_\vi$ has at least one 
good subgraph $B^{(l+1)}_\vv$ satisfying $\fp(\vv)=\vi$, 
at most two touched subgraphs  
$B^{(l+1)}_\vv, B^{(l+1)}_\vu$ satisfying $\fp(\vv)=\fp(\vu)=\vi$, 
$\nu^{(l+1)}_\vv  \ge \underline \nu^{(l+1)}$  if  
$B^{(l+1)}_\vv$ is good, and
$\nu^{(l+1)}_\vv  \le \bar \nu^{(l+1)}$ if  $B^{(l+1)}_\vv$ is touched.
 
Now fix $\vi\in\vI_l$.  $S^{(l)}_\vi$ consists of 
$n_l/n_{l+1}$ many slabs of level $l+1$, namely 
$S^{(l+1)}_{\vi,i_{l+1}}, i_{l+1} \in [n_l/n_{l+1}]_0$. 
Each such slab has at least one good subgraph by 
\corO{Assumption (I) for $k=l+1$}.
Also, as mentioned in the \corO{beginning of the} proof, the portion of 
$\vpi$ within slab $S^{(l)}_\vi$ resides in either 
one subgraph $B^{(l)}_\vv\in\cB_l$ 
satisfying $\fp(\vv)=\vi$ or in two (consecutive   and disjoint) subgraphs $B^{(l)}_\vv, B^{(l)}_\vu\in\cB_l$ satisfying $\fp(\vv)=\fp(\vu)=\vi$.    
In the first case,  it is obvious that $B^{(l)}_\vv$  is good.  
In the second case, if both $B^{(l)}_\vv$ and $B^{(l)}_\vu$ are not good, 
then there will be at least one slab among 
$S^{(l+1)}_{\vi, i_{l+1}}, i_{l+1} \in[n_l/n_{l+1}]_0,$   
having no good subgraph of level $l+1$ within it. 
This leads to a contradiction to the \corO{induction hypothesis concerning
(I). We conclude that (I) holds for all $k\in [K+1]_0$.}

\corO{We now turn to the proof of (A),(B). Again, the proof is by backward
induction on $k$. We begin with proving the basis of the induction.}
Suppose \corO{that}
$B^{(K)}_\vv$ is good. Then, using Proposition \ref{Quad}, 
particularly the lower bound in \eqref{nuGm def},   we get that
$\nu^{(K)}_{\corO{\vv}} \ge c\mu^2 \sqrt{\log(n_K)}$, 
\corO{thus showing (A) in case $k=K$}. 
On the other hand, the upper bound in  
\eqref{nuGm def} implies $\nu^{(K)}_{\vv,\vpi} \le c^{-1} \mu^2 
\sqrt{\log(n_K)}$ whenever $\vpi$ intersects $B^{(K)}_\vv$, 
\corO{showing (B) in case $k=K$.}
\corO{This completes the proof of the base of the induction.}

Now suppose \corO{(A),(B)}
hold for  $k=l+1$.  
\corO{Suppose $B^{(l)}_\vv$ is good.}  
Then, using Proposition \ref{Quad}, particularly the 
lower bound in \eqref{nuGm def}, and noting that 
$\bar B^{(l)}_\vv$  resides within either a square 
or an isosceles triangle having side length $n_l/n_{l+1}$,
$\nu^{(l)}_{\corO{\vv}} \ge c(\underline \nu^{(l+1)})^2 
\sqrt{\log(n_l/n_{l+1})}$.  On the other hand, 
the upper bound in \eqref{nuGm def} implies 
$\nu^{(l)}_{\corO{\vv}} \le \frac 1c(\bar \nu^{(l+1)})^2 \sqrt{\log(n_l/n_{l+1})}$ whenever $\vpi$ intersects $B^{(l)}_\vv$.
Combining these with the expressions of 
$\underline \nu^{(l+1)}$ and $\bar \nu^{(l+1)}$  
(obtained from \eqref{ind hyp}) we see that (A) and (B) of \eqref{ind hyp} hold for $k=l$. Thus, all assertions of \eqref{ind hyp} are true for $k=l$, and the induction argument is complete. 
\end{proof}

\corO{We return to the proof of \corOO{the upper bound of} 
   Theorem \ref{UnknownInitial}. Note first that the choice of
   constants $K$ and $\epsilon_k$ made in \eqref{eq-021617a} and
   \eqref{eq-021617b} below, together with Lemma \ref{lem-CD}(B), 
 ensure that 
$\max_{\vu\in\bar V^{(k)}_\vv} (\nu^{(k)}_{\corO{(\vv,\vu)}})^2\leq 1$.
 }
 By construction it \corO{follows, using Proposition \ref{Quad},}
 that for all $0\le k\le K$,
\[ Q^{(k)}_\vv \eqd
\begin{cases}
W^{(k)}_\vv  & \text{ under } H_0 \\
W^{(k)}_\vv  + U^{(k)}_{\corO{\vv}}  + \nu^{(k)}_{\corO{\vv}} & 
\text{ under } H_{1,\vpi} 
\end{cases},  \E  U^{(k)}_{\corO{\vv}} =0, 
\E[ U^{(k)}_{\corO{\vv}} ]^2 \le C\frac{1}{\log(n_k/n_{k+1})} 
\max_{\vu\in\bar V^{(k)}_\vv} (\nu^{(k+1)}_{\corO{(\vv,\vu)}})^2.
\]
In particular, $\pr_0(Q^{(0)}_\emptyset \le \cdot) = \pr(W^{(0)}_\emptyset \le \cdot)$, so using  \eqref{Wmbd} with \corO{$n$}
replaced by $|\vV_1|\asymp (n_0/n_1)^2$,
\[
\sup_{x\in\dR} |\pr_0(Q^{(0)}_\emptyset \le  x) - \gF(x)| \le C \left[ \left(\log n^{2-2\eps_1}\right)^{-1/2}+\left(\frac{\gb_3^{(1)}}{n^{2-2\eps_1}}+\frac{(\gb_3^{(1)})^2}{n^{2-2\eps_1}\log^{3/2}n^{2-2\eps_1}}\right)^{1/4}\right],\]
where $\gb_3^{(1)} := \max_{\vv\in\vV_1} \E(|W^{(1)}_\vv|^3)$. In order to bound $\gb_3^{(1)}$, we will use Theorem \ref{Quad Moment Bd}. If we define $\gb_s^{(k)} := \max_{\vv\in\vV_k}\E(|W^{(k)}_\vv|^s)$,  
then Theorem \ref{Quad Moment Bd} \corO{gives} 
$\gb_s^{(k)} \le \corO{C_s}
\gb_{2s}^{(k+1)}$ for all $k\in[K-1]$, 
\corO{where $C_s=2^{5s/2}\Gamma((s+1)/2)\Gamma(s+1/2)^{1/2}$}. 
Also $\gb_s^{(K)}\asymp \int_\dR |x|^s\; d\gF(x) \asymp \gC((s+1)/2)$. Combining \corO{the}
last two observations,
\[ \gb_3^{(1)} \le \exp\left[\frac 52\log 2 \sum_{l=0}^{K-1} 3 \cdot 2^l +\sum_{l=0}^{K} \log\gC\left(\frac 12+\frac 32 2^l\right) + \frac 12\sum_{l=0}^{
\corO{K-1}} 
\log\gC\left(\frac 12+ 3\cdot 2^l\right)\right] 
\le \exp\left(CK^2 2^K\right) \]
for some $C>0$. The last inequality holds because $\gC(k) \asymp k^k$.  
Combining the last two displays \corO{we get}
\beq\label{W0bd}
\sup_{x\in\dR} |\pr_0(Q^{(0)}_\emptyset \le  x) - \gF(x)| \le C \left[ \left(\log n^{2-2\eps_1}\right)^{-1/2}+ \exp(CK^22^K-\frac 12(1-\eps_1) \log n)\right].
\eeq
for some $C>0$.

\corO{We now specify the constants $K, \epsilon_k$ and
  simplify the formula for
  $\underline \nu^{(0)}_\emptyset, \bar\nu^{(0)}_\emptyset$:}
\beqax 
\underline\nu^{(0)}_\emptyset &=& 
 \frac 1c \exp\left[2^{K+1}\log(c\mu)\right] \prod_{l=0}^{K} [\log(n_l/n_{l+1})]^{2^{l-1}} \\
& = & \frac {1}{c\sqrt{\log n}} \exp\left[2^{K+1}\log(c\mu\sqrt{\log n})\right] \prod_{l=0}^{K} [\eps_1\eps_2\cdots\eps_l(1-\eps_{l+1})]^{2^{l-1}}.
\eeqax
Rearranging the terms, the last product equals 
\[ \prod_{s=1}^K  (1-\eps_s)^{2^{s-2}} \eps_s^{2^{K}-2^{s-1}} = \prod_{s=1}^K \left[(1-\eps_s)\eps^{2^{K-s+2}-2}\right]^{2^{s-2}}. \] 
It is not difficult to see that the above product will be maximized if we take 
\begin{equation}
  \label{eq-021617a}
  1-\eps_s=(2^{K-s+2}-1)^{-1}, \quad  s\in[K]. 
\end{equation}
\corO{With this choice,}
  and using the fact that $e^{-k\eps} \ge (1-\eps)^k \ge 1-k\eps$ 
  for any $k>0$, we see that the product is 
\[ \asymp \exp\left(  -\log 2 \sum_{s=1}^K (K-s+2)2^{s-1} \right)
= \exp\left(  -\log 2 \sum_{s=2}^{K+1} s2^{K-s+1} \right) = \exp(-C_12^K) \]
for some constant $C_1>0$. Therefore, there are constants $c, C_1>0$ such that
\[ \nu^{(0)}_\emptyset \ge  \frac {1}{c\sqrt{\log n}} \exp\left[2^{K+1}\left(\log(c\mu\sqrt{\log n}) - C_1\right)\right] \]
Now applying (2) of Proposition \ref{Quad},
\beqax 
\E[(U^{(0)}_\emptyset)^2] &\le & C_2\frac{(\bar\nu^{(1)})^2}{\log(n/n_1)} = C_2c\frac{\bar\nu^{(0)}}{\log^{3/2}(n/n_1)} \\
&\le & C_2c \frac{2^{\frac 32(K+1)} }{\log^2 n}\exp\left[2^{K+1}\left(\log(\frac 1c \mu\sqrt{\log n}) - C_1\right)\right] \eeqax
We can choose $\gD_0>0$ large enough so that for any $\gD\ge\gD_0$, $\log(c\gD)-C_1>0$ and $\log(\frac 1c\gD)-C_1\le \frac 54 \log(c\gD)-C_1$. If we assume $\mu\sqrt{\log n}\ge \gD_0$, and choose $K$ so that 
\begin{equation}
  \label{eq-021617b}
  2^{K+1}[\log(c\mu\sqrt{\log n})-C_1] =\log\log n,
\end{equation}
then $\nu^{(0)}_\emptyset \ge \sqrt{\log n}\corO{/c}$, $\E[(U^{(0)}_\emptyset)^2]=o(1)$ and the bound in \eqref{W0bd} is $o(1)$. So for this choice of $K$, $W^{(0)}_\emptyset \convD N(0,1), U^{(0)}_\emptyset \convP 0$ and $\nu^{(0)}_\emptyset \to \infty$. So if one rejects the null hypothesis when 
\corO{$Q^{(0)}_\emptyset$
exceeds $\underline\nu^{(0)}_\emptyset/2$,}
then its minimax risk will be $o(1)$. 
\end{proof}

\subsection{Proof of the lower bound of detection threshold}
\begin{proof}[Proof of Theorem \ref{UnknownInitial}, lower bound.]
\corO{Since the hypothesis testing problem $(\sP(\sG_n^{(a)}),\mu_n,\gF)$ has a signal
hypothesis which is strictly larger than that of 
$(\sP(\sG_n^{(0)}),\mu_n,\gF)$,
the asymptotically powerless part of Theorem \ref{UnknownInitial}
follows from \cite[Theorem 1.1]{ACHZ08}.}
\end{proof}

\section{Proof of the supporting results}
\subsection{Proof of Lemma \ref{A:properties}} \label{sec-quadlemproofs}
\begin{proof}[Proof of Lemma \ref{A:properties}]
(1).  Note that $\norm{[\vA(\sV_n)]}_F^2=2\sum_{(i,u) \in \sV_n} \;\;\sum_{(j,v): (i,u) \precsim (j,v)}    (j-i)^{-2}$. In order to estimate the above sum note that for fixed $i,u,j$ the number of choices of $v$ satisfying $(i,u)\precsim (j,v)$ is $\asymp |j-i|$. Combining this with the fact that $\sum_{1\le k\le i} k^{-1} \asymp \log i$ implies that 
\beq \label{var0}
 \norm{[\vA(\sV_n)]}_F^2 \asymp \sum_{i,u,j: (i,u)\in\sV_n \text{ and } j>i} (j-i)^{-1} \asymp \sum_{i,u: (i,u)\in\sV_n} \log(n-i) \asymp n^2\log n.\eeq

\noindent
(2).  In order to bound the spectral 
norm of $[\vA(\sV_n)]$ we will use the well known fact that 
\corO{for any matrix $\vB\in\dR^{k\times l}$,}
 \[ \max\{|\gl|: \text{ $\gl$ is an eigenvelue of $\vB$}\} \le \max_{i\in[k]}\sum_{j\in[l]} |b_{i,j}|.\]
In our case, since $[\vA(\sV_n)]$ is symmetric,  $\norm{[\vA(\sV_n)]}$ 
equals the largest absolute eigenvalue \corO{of $[\vA(\sV_n)]$}, so
$\norm{[\vA(\sV_n)]} \le \max_{(i,u)\in\sV_n} \sum_{(j,v) \in \sV_n} [A(\sV_n)]_{(i,u),(j,v)}$.   
Note that for each $(i,u)\in\sV_n$, 
\[ \sum_{(j,v)\in\sV_n} [A(\sV_n)]_{(i,u),(j,v)} = \sum_{j\in [n], j \ne i} \sum_{v: (j,v)\in \llb(i,u)\rrb} \frac{1}{|i-j|} \asymp \sum_{j\in [n], j \ne i} 1 \asymp n, \] 
which gives the result.

\noindent
(3).   For any $\pi\in\sP_n$ and $i\in[n]_0$, let $\pi_i\in\dZ$ be such that $(i,\pi_i)\in\pi$.  Then
\beqax
\mathds 1_\pi^T [\vA(\sV_n)]\mathds 1_\pi = 2\sum_{0\le i<j<n } [\vA(\sV_n)]_{(i,\pi_i),(j,\pi_j)} 
= 2\sum_{0\le i<j<n } (j-i)^{-1}
\asymp \sum_{i\in[m]_0}  \log(n-i) \asymp n\log n.
\eeqax

\noindent
(4). For $\pi\in\sP_n$, $\mathds 1_\pi^T [\vA(\sV_n)]^2\mathds 1_\pi$ equals
\beqax  \sum_{i, j\in[n]_0} \left(\left[\vA(\sV_n)\right]^2\right)_{(i,\pi_i),(j,\pi_j)} 
&=& 2\sum_{0\le i< j<n}\;\; \sum_{(k,v) \in \llb(i, \pi_i)\rrb \cap \llb(j,\pi_j)\rrb} (|i-k|\cdot|j-k|)^{-1}\\
&& +  \sum_{i\in[n]_0}\;\; \sum_{(k,v) \in \llb(i, \pi_i)\corO{\rrb}} (k-i)^{-2}  \;\; =: \;\; I_1+I_2. 
\eeqax
In order to estimate $I_2$ note that for any $i,k \in [n]_0$, $|\{v: (k,v)\lrs  (i,\pi_i)\}| \asymp |k-i|$. So
\beq \label{I2bd} 
I_2 \asymp \sum_{i\in[n]_0}\sum_{k\in[n]_0, k\ne i} |k-i|^{-1} \asymp n\log n.\eeq
 In order to estimate $I_1$ note that  the $(i,j)$-th inner sum of $I_1$ equals
\begin{align} 
 &  \left[ \sum_{k=0}^{i-1}  \;\; \sum_{v: (k,v) \lrs  (i,\pi_i), (j,\pi_j)} (i-k)^{-1}(j-k)^{-1} 
   +\sum_{k=j+1}^{n-1} \;\;
 \sum_{v: (k,v) \in (i,\pi_i), (j,\pi_j)} (k-i)^{-1}(k-j)^{-1} 
 \right] \notag \\
 &  + \left.
  \sum_{k: i< k<j} \sum_{v: (k,v) \lrs  (i,\pi_i), (j,\pi_j)} (k-i)^{-1}(j-k)^{-1} \right] \;\; =: \;\; I^{i,j}_{1,1} + I^{i,j}_{1,2} + I^{i,j}_{1,3}. \label{Iij11def}
\end{align}
 For the sum $I^{i,j}_{1,1}$, we see that  $(k,v)\lrs (i,\pi_i)$ implies $(k,v)\lrs (j,\pi_j)$. So the number of summands in the $k$-th inner sum is $\asymp (i-k)$. Hence, 
 \begin{equation}
   \label{eq-11.1}
   I^{i,j}_{1,1}\asymp \sum_{k=0}^{i-1} (j-k)^{-1} \asymp \int_{j-i}^j \frac 1x\; dx=\log\frac{j}{j-i}, 
 \end{equation}
 and so 
 \beqax
 \sum_{0\le i<j<n}  I^{i,j}_{1,1} &\asymp& \sum_{0\le i<j<n} \log j - \sum_{0\le i<j<n} \log(j-i)
    = \sum_{j=1}^{n-1} j\log j  \\
     && \hspace*{.5cm} - \sum_{i=0}^{n-2} \sum_{k=1}^{n-1-i} \log k = \sum_{j=1}^{n-1} j\log j - \sum_{k=1}^{n-1} (n-k)\log k = \sum_{j=1}^{n-1} 2j \log j - n    \sum_{j=1}^{n-1} \log j \\
     &\asymp& \int_1^n 2x\log x\; dx - n\int_1^{n} \log x\; dx = n^2\log n -(n^2-1)/2 - n^2\log  n+ n(n-1) \asymp n^2.
 \eeqax
 Using a similar argument,
 $I^{i,j}_{1,2} \asymp \sum_{k=j+1}^{n-1} (k-i)^{-1} \asymp \log\frac{n-i-1}{j_i}$, so 
 \corO{$\sum_{0\le i<j<n}  I^{i,j}_{1,2}\asymp n^2$.} 
  For the sum $I^{i,j}_{1,3}$, the number of summands in the inner sum is at most $2(k-i)$ (resp.~$2(j-k)$) when $k<(i+j)/2$ (resp.~$k\ge(i+j)/2$). 
  Thus 
  \begin{equation}
    \label{eq-11.5}
    \corO{I^{i,j}_{1,3}}\leq \sum_{k: i+1 \le k \le (i+j)/2} (j-k)^{-1}
  + \sum_{k:(i+j)/2< k < j} (k-i)^{-1}  \asymp \int_{(j-i)/2}^{j-i} \frac 1x \; dx \asymp 1,
\text{ so } \sum_{0\le i<j<n} I^{i,j}_{1,3} \le Cn^2 
\end{equation}
for some constant $C$. Combing the last three displays, $I_1 =\sum_{k=1}^3\sum_{0\le i<j<n} I^{i,j}_{1,k} \asymp n^2$. This together with \eqref{I2bd} gives the result. 

\noindent
(5). For $\pi\in\sP_n$, $\ind_{\vpi}^T [\vA(\sV_n)]\textrm{Diag}(\ind_{\vpi})[\vA(\sV_n)]\ind_{\vpi}$ equals
\beqax  \sum_{i, j\in[n]_0} \left(\left[\vA(\sV_n)\right]\textrm{Diag}
(\ind_{\vpi})[\vA(\sV_n)]\right)_{(i,\pi_i),(j,\pi_j)} 
&=& 2\sum_{0\le i< j<n}\;\; \sum_{k \in [n]_0, k\ne i, j} (|i-k|\cdot|j-k|)^{-1}\\
&& +  \sum_{i\in[n]_0}\;\; \sum_{k \in [n]_0, k\ne i} (k-i)^{-2}  \;\; =: \;\; 
\corO{J_1+J_2}. 
\eeqax
In order to estimate $J_2$ note that for any $i \in [n]_0$, $\sum_{k\in[n]_0, k\ne i} (k-i)^{-2} \asymp C$, so $J_2 \asymp n$.
 In order to estimate $J_1$ note that  the $(i,j)$-th inner sum of $J_1$ equals
\begin{align*} 
 &  \left[ \sum_{k=0}^{i-1}  \;\; (i-k)^{-1}(j-k)^{-1} 
   +\sum_{k=j+1}^{n-1} \;\;
  (k-i)^{-1}(k-j)^{-1} 
 \right]\\
 &  + \left.
  \sum_{k: i< k<j}  (k-i)^{-1}(j-k)^{-1} \right] \;\; =: \;\; J^{i,j}_{1,1} + J^{i,j}_{1,2} + J^{i,j}_{1,3}.
\end{align*}
 For the sum $J^{i,j}_{1,1}$, we see that  
 $$J^{i,j}_{1,1}= \sum_{k=0}^{i-1} \frac{1}{j-i}[(i-k)^{-1}-(j-k)^{-1}] \asymp \frac{1}{j-i}[\int_{1}^i \frac 1x\; dx-\int_{j-i}^j \frac 1x\; dx]=\frac{1}{j-i}\log\frac{i(j-i)}{j},$$ and so 
 \beqax
 \sum_{0\le i<j<n}  J^{i,j}_{1,1} &\asymp& \sum_{0\le i<n} \log i  \sum_{j:i<j<n}\frac{1}{j-i} -\sum_{0<j<n} \log j \sum_{0\le i<j}\frac{1}{j-i} + \sum_{0\le i<n} \sum_{j-i=1}^{n-i}\frac{\log(j-i)}{j-i}
    \\
    &\asymp& \sum_{0\le i<n} \log i\log(n-i) - \sum_{0<j<n} (\log j)^2 +  \sum_{0\le i<n} \frac 12 (\log(n-i))^2\\
     &\le&  \log n\int_1^n \log x\; dx - \frac 12 \int_1^{n} (\log x)^2\; dx = \frac 12 n(\log n)^2.
 \eeqax
 Using \corO{a similar argument,}
 $$J^{i,j}_{1,2} \asymp \sum_{k=j+1}^{n-1} \frac{1}{j-i}[(k-j)^{-1}-(k-i)^{-1}] \asymp \frac{1}{j-i}\log\frac{(n-j)(j-i)}{n-i-1}$$ 
 \corO{and therefore   
 $ \sum_{0\le i<j<n}  J^{i,j}_{1,2}\asymp n(\log n)^2/2$.}
 For the sum $J^{i,j}_{1,3}$, \corO{we have that}
 \[  J^{i,j}_{1,3}\leq 
 \sum_{k: i+1 \le k \le j-1} \frac{1}{j-i}[(k-i)^{-1}+(j-k)^{-1}]
    \asymp \frac{1}{j-i} \int_{1}^{j-i} \frac 1x \; dx \asymp  \frac{\log(j-i)}{j-i}, \]
    \corO{and therefore}
    \[ \sum_{0\le i<j<n} J^{i,j}_{1,3} =\sum_{0\le i<n} \sum_{j-i=1}^{n-1-i} \frac{\log(j-i)}{j-i} \asymp  \sum_{0\le i<n} \int_1^{n-i} \frac{\log x}{x}\; dx \asymp \sum_{0\le i<n} (\log(n-i))^2 \asymp n(\log n)^2. \]
 Combing the last 
 \corO{estimates, we obtain that}
 $J_1 =\sum_{k=1}^3\sum_{0\le i<j<n} J^{i,j}_{1,k} \asymp \corO{n(\log n)^2}$. 
 This together with the bound on $J_2$  gives the result. 

\noindent
(6). Note that $\norm{[\vA(\sV_n)]^2}_F^2$ equals
\beqax  
\sum_{\vu,\vv\in\sV_n} \left(\left(\left[\vA(\sV_n)\right]^2\right)_{\vu,\vv}\right)^2 
&\asymp& n^2\left[\sum_{0\le u_1< v_1<n}\left( \sum_{\vw \in \llb\vu\rrb \cap \llb\vv\rrb} (|u_1-w_1|\cdot|v_1-w_1|)^{-1}\right)^2\right. \\
  && +  \left.\sum_{u_1\in[n]_0}\left(\sum_{\vw \in \llb\vu\rrb} (w_1-u_1)^{-2}\right)^2\right]  \;\; =: \;\; \corO{K_1+K_2}. 
\eeqax
In order to estimate $K_2$ note that for any $u_1, w_1 \in [n]_0$, $|\{w_2: \vw\lrs  \vu\}| \asymp |w_1-u_1|$. So
\[
K_2 \asymp  n^2 \sum_{u_1\in[n]_0}\left(\sum_{w_1\in[n]_0, w_1\ne u_1} |w_1-u_1|^{-1}\right)^2 \asymp n^3 \log^2n.\]
Next note that  the $(u_1,v_1)$-th inner sum of $K_1$ equals $n^2$ times
 \[ (I_{1,1}^{u_1,v_1} + I_{1,2}^{u_1,v_1} + I_{1,3}^{u_1,v_1})^2 
   \corO{\asymp}
    \left[(I_{1,1}^{u_1,v_1})^2 + (I_{1,2}^{u_1,v_1})^2 +( I_{1,3}^{u_1,v_1})^2\right],\]
where $I_{1,i}^{u_1,v_1}, i=1,2,3,$ are as in \eqref{Iij11def}. 
 \corO{We have
   that $I^{u_1,v_1}_{1,1}\asymp\log\frac{v_1}{v_1-u_1}$, see \eqref{eq-11.1},}
   and therefore 
 \beqax
 \sum_{0\le u_1<v_1<n}  \left(I^{u_1,v_1}_{1,1}\right)^2 &\asymp& \sum_{0\le u_1<v_1<n} \log^2v_1 - \sum_{0\le u_1<v_1<n} \log^2(v_1-u_1)\\
    &=& \sum_{j=1}^{n-1} j\log^2j  
      - \sum_{i=0}^{n-2} \sum_{k=1}^{n-1-i} \log^2k
      = \sum_{j=1}^{n-1} j\log^2j - \sum_{k=1}^{n-1} (n-k)\log^2k\\ 
      &=& \sum_{j=1}^{n-1} 2j \log^2j - n    \sum_{j=1}^{n-1} \log^2j 
      \asymp \int_1^n 2x\log^2x\; dx - n\int_1^{n} \log^2x\; dx \\&=& n^2(\log^2n - \log n) -\frac{n^2-1}{2} - n^2(\log^2n-2\log n)+ 2n(n-1) \asymp n^2\log n.
 \eeqax
 Using similar argument $\sum_{0\leq i<j<n}
 (I^{i,j}_{1,2})^2 \asymp n^2\log n$.   
  \corO{We also have that 
  $\sum_{0\le u_1<v_1<n} (I^{u_1,v_1}_{1,3})^2 \leq C n^2$, see
  \eqref{eq-11.5}}.  Combing the \corO{last} 
  estimates, $I_1$ and hence $\norm{[\vA(\sV_n)]^2}_F^2$ is 
  \corO{$\asymp n^4\log n$}. This together with 
  \corO{Lemma \ref{A:properties}(1)} 
  gives the result. 
\end{proof}

\subsection{Proof of Proposition \ref{Quad Kolmogorov Bd}} \label{Quad K Bd proof}
\begin{proof}[Proof of Proposition \ref{Quad Kolmogorov Bd}]
  Using Theorem \ref{IID Quad Bd} and   \corO{Lemma \ref{A:properties}(2)},  
\beq \label{gaussianbd} 
\sup_{x\in\dR} \left|\pr\left(G_n([\bar\vA(\sV_n)]) \le x\right) - \gF(x)\right| \le C(\log n)^{-1/2}. \eeq
We obtain $\nu_\vv, s_\vv^2, \gD$ and $L$ (as defined in \eqref{LDelta def}) for $[\bar\vA(\sV_n)]$ and $(U_\vv, \vv\in \sV_n)$. 
\corO{By Lemma \ref{A:properties}(6) we have}
$\gD\asymp (\log n)^{-1/2}$, so we can apply
heorem \ref{IND Quad Bd}. Suppose $F_\vv$ denotes the distribution 
function  of $U_\vv$ and let $\bar F_\vv:=1-F_\vv$ and $\bar \gF:=1-\gF$.
\corO{Then},
\beqax 
\nu_\vv  &\le& \int_0^\infty 3x^2 (\bar F_\vv(x) + \bar \gF(x)) \; dx  + \int_{-\infty}^0 3x^2 (F_\vv(x) +\gF(x))\; dx  \\
&=& \int_0^\infty (\bar F_\vv(x)+\bar\gF(x) + F_\vv(-x)+\gF(-x)) \; d(x^3) \\
&=& \int_0^\infty (\pr(|X_{F_\vv}|^3>y)  + \pr(|X_\gF|^3>y) )\; dy =\int_\dR|x|^3\; d(F_\vv+\gF)(x)
\le \gb_3+4/\sqrt{2\pi}.
\eeqax
\corO{We have that
  $\sum_{\vu\in\sV_n} [\vA(\sV_n)]_{\vv,\vu}^2 \asymp \log n$ and $\sum_{\vu\in\sV_n} [\vA(\sV_n)]_{\vv,\vu}^3 \asymp 1$ for each $\vv\in\sV_n$,
  by the argument leading to \eqref{var0}.}
This together with Lemma \ref{A:properties}(1) gives 
$s_\vv^2 \asymp n^{-2},   \sum_{\vu,\vv\in\sV_n} [\bar\vA(\sV_n)]_{\vv,\vu}^3   \asymp \frac{1}{n\log{3/2}n}$. 
These estimates give  $L \le \frac Cn (1+\gb_3+\frac{\gb_3^2}{\log^{3/2}n})$. 
Using that,
Theorem \ref{IND Quad Bd} and \eqref{gaussianbd} give the claimed bound.
\end{proof}

\subsection{Proof of Proposition \ref{Quad}} \label{QuadProof}
 We begin with the following observation.
\begin{lem} \label{nlog n bd}
$\min_{I\subset[n]} \max_{S\in\{I,[n]\setminus I\}}\sum_{i,j\in S, i\ne j} |j-i|^{-1} \ge c n\log n$ for some constant $c>0$.
\end{lem}  

\begin{proof}
Let $A_n:=\min_{I\subset[n]} 
\max_{S\in\{I,[n]\setminus I\}}\sum_{i,j\in S, i\ne j} |j-i|^{-1}$. 
We have that
\begin{eqnarray}
  \label{eqYK1}
  A_n&\geq & \frac12 \min_{I\subset[n]}
  \left(\sum_{i,j\in I, i\ne j} |j-i|^{-1}+
  \sum_{i,j\in [n]\setminus I, i\ne j} |j-i|^{-1}\right) \nonumber\\
 &\geq &
 \frac12 \min_{I\subset [n], |I|\geq n/2}\sum_{i,j\in I, i\ne j} |j-i|^{-1}
 \nonumber \\
 &=& \frac12 \min_{\alpha\in [1/2,1]} \min_{I\subset [n], |I|= \lfloor\alpha n\rfloor}
 \sum_{i,j\in I, i\ne j} |j-i|^{-1}=: \frac12 \min_{\alpha\in [1/2,1]}
 B_n(\alpha)\,.
 \end{eqnarray}
 Since $\alpha \mapsto B_n(\alpha)$ is monotone, the lemma will follow
 from \eqref{eqYK1} if we show the existence of a constant $c>0$ so that, for all $n$ integer, 
 \begin{equation}
   \label{eqYK2}
   B_n(1/2)\geq cn\log n.
 \end{equation}

 To prove \eqref{eqYK2}, we begin by claiming that there exists $\alpha_0<1$
 and a constant $c>0$ 
 so that 
 \begin{equation}
   \label{eqYK3}
   B_n(\alpha_0)>cn\log n. 
 \end{equation}
 Indeed, with $|I|=\alpha n$,
 \begin{eqnarray*}
\sum_{i,j\in I, i\ne j} |j-i|^{-1}&=&
\sum_{i,j\in [n], i\neq j} |j-i|^{-1}-
\sum_{i,j\in [n]\setminus I, i\neq j} |j-i|^{-1}-
2\sum_{i\in [n]\setminus I, j\in I} |j-i|^{-1}\\
&\geq & n\log n(1+o(1))-2\sum_{i\in [n]\setminus I, j\in [n], j\neq i}
|j-i|^{-1}\\
&= & (1-4(1-\alpha))n\log n(1+o(1)).
\end{eqnarray*}
In particular, \eqref{eqYK3} holds with $\alpha_0=4/5$ and $c=1/5$.
Note that unfortunately, we can't yet take $\alpha_0=1/2$.

Continuing with the proof of \eqref{eqYK2}, note that by rescaling,
for any 
$\alpha\in (0,1)$,
\begin{equation}
  \label{eqYK4}
  B_{n/2}(\alpha)=2 B_n(\alpha/2).
\end{equation}
Thus, with $\beta_0:=\alpha_0/2<1/2$,
\[ B_n(\beta_0)= \frac12 B_{n/2}(2\beta_0)=
\frac12 B_{n/2}(\alpha_0)\geq \frac{c}{4} n\log n.\]
By the monotonicity of $\alpha\mapsto B_n(\alpha)$, this proves 
\eqref{eqYK2} and hence the lemma.
\end{proof}

Combining Lemma \ref{nlog n bd} with Lemma \ref{A:properties} we prove Proposition \ref{Quad}.

\begin{proof}[Proof of Proposition \ref{Quad}]
\corO{To minimize notation, we set $\sG_{\corS{n}}=\sG_{\corS{n}}^{(a)}$, and let all constants
depend implicitly on $a$.} 

\noindent 
 (1) Using the partial order $\preccurlyeq$, we order the random variables $\{Z_\vv: \vv \in \sV_n\}$ to have the $|\sV_n| \times 1$ column vector $\vZ_n$.  
 \corO{Set $\vR_{\vPi}=\sum_{\vv\in\vPi} Y_{\vv,\vPi} \ind_{\{\vv\}}$.}
Recalling the fact that $[\vA(\sV_n)]$ has zero diagonal entries,
 \beq\label{meaneq} 
 \E \vZ_n^T[\vA(\sV_n)]\vZ_n=0, \quad \E \vZ_n^T[\vA(\sV_n)]\vR_{\vPi}=0, \quad \E\vR_{\vPi}^T[\vA(\sV_n)]\vR_{\vPi}=0,\eeq
 because each summand of all the above quadratic forms  is product of independent random variables having mean zero. Observing that the summands $\vZ_n^T[\vA(\sV_n)]\vZ_n$  are uncorrelated, and using (1) of Lemma \ref{A:properties},
\[
\E[(\vZ_n^T[\vA(\sV_n)]\vZ_n)^2]= 4\sum_{(i,u)  \in \sV_n} \;\;\sum_{(j,v)\in \sV_n: (i,u) \precsim (j,v)}    (j-i)^{-2}  
 =  2\norm{[\vA(\sV_n)]}_F^2 \asymp n^2\log n.  \]
Since $W_n$ is the normalized version of  $\vZ_n^T[\vA(\sV_n)]\vZ_n$, $W_n$ has mean 0 and variance 1. 
Invoking Proposition \ref{Quad Kolmogorov Bd} we get  the desired bound.

 \noindent
(2) Using the partial order $\preccurlyeq$,  define  $|\sV_n| \times 1$ column vector $\vX_n=(X_\vv, \vv \in \sV_n)$. Also define  $\vgY_{\vPi} := \sum_{\vv\in\vPi} \nu_{\vv,\vPi} \ind_{\{\vv\}}$.
Clearly
\begin{align*}
&          \vX_n \eqd \begin{cases}
                       \vZ_n & \text{ under } \pr_0 \\
                       \vZ_n +  \vR_{\vPi} + \vgY _{\vPi}  & \text{ under } \pr_{1,\vPi}
                         \end{cases},  \text{ so }
          Q_n \eqd \begin{cases}
                       W_n & \text{ under } \pr_0 \\
                       W_n +  U_n + \nu(\sG_n)  & \text{ under } \pr_{1,\vPi}
                         \end{cases}, \notag \\
 &  \text {where }
 U_n := 2\vgY_{\vPi}^T[\bar\vA(\sV_n)]\vZ_n+2\vgY_{\vPi}^T[\bar\vA(\sV_n)]\vR_{\vPi} +2\vR_{\vPi}^T[\bar\vA(\sV_n)]\vZ_n + \vn_{\vPi}^T[\bar\vA(\sV_n)]\vR_{\vPi}  \\
 & \text{ and } h(\vPi,\nu_n) :=\vgY_{\vPi}^T[\bar\vA(\sV_n)]\vgY_{\vPi}  
    \end{align*}
  It follows from \eqref{meaneq} and the facts $\E\vZ_n= \E \vR_{\vPi}=\mathbf 0$ that each summand of $U_n$ has mean 0, so $\E U_n = 0$.   Now we observe that 
 \begin{enumeratea}
 \item $[\bar\vA(\cdot)]$ has zero diagonal entries .
 \item $\E[\vZ_\vv W_n] =\E[Y_{\vv,\vPi} W_n] =0$ for all $\vv\in\sV_n$, as components of $\vZ_n$ are independent  with mean 0.
 \item $\E[\vZ_\vv\vZ_{\vv'}Y_{\vu,\vPi}Y_{\vu',\vPi}] = \E[\vZ_\vv\vZ_{\vv'}\vZ_{\vu}Y_{\vu',\vPi}=0$ for all $\vv, \vv', \vu, \vu' \in \sV_n$ satisfying $\vv\ne\vv', \vu\ne\vu'$, as $\{(\vZ_\vv,Y_{\vv,\vPi}): \vv\in\sV_n\}$ are independent, and $\vZ_\vv$ and $Y_{\vv,\vPi}$ are uncorrelated for all $\vv\in\sV_n$.
   \item  $\E[\vZ_n\vZ_n^T] = I_{|\sV_n|}$ (the identity matrix) , $\E[\vR_{\vPi}\vR_{\vPi}^T] \leqq \textrm{Diag}(\ind_{\vPi})$.
 \end{enumeratea}
 Using these observations, each of the summands of $ W_nU_n$ has mean zero, so $\E W_nU_n=0$.  Also
 \beqax 
 \E[(\vR_{\vPi}^T[\bar\vA(\sV_n)]\vZ_n)^2] 
 &=& \frac 12 \norm{[\vA(\sV_n)]}_F^{-2} \sum_{\vv\in\vPi} \sum_{\vv'\in\llb\vv\rrb} \E Y_{\vv,\vPi}^2 \E\vZ_{\vv'}^2(A_{\vv,\vv'})^2 \\
 \E[(\vR_{\vPi}^T[\bar\vA(\sV_n)]\vR_{\vPi})^2] 
 &=& \frac 12 \norm{[\vA(\sV_n)]}_F^{-2} \sum_{\vv\in\vPi} \sum_{\vv'\in\llb\vv\rrb\cap\vPi} \E Y_{\vv,\vPi}^2 \E Y_{\vv',\vPi}^2(A_{\vv,\vv'})^2
 \eeqax
 Using Lemma \ref{A:properties}(1), both terms in the above display are
 \beqax
 &\le&  \norm{[\vA(\sV_n)]}_F^{-2} \sum_{0\le i<j<n}\sum_{v:(i,u)\lrs(j,v) \text{ for some } u\in\vPi_i} (j-i)^{-2}  \\
 &\asymp& \frac{1}{n^2\log n}  \sum_{0\le i<j<n} (j-i)^{-1}
 \asymp \frac{1}{n^2\log n}  \sum_{0\le i<n-1} \log(n-i)
 \asymp \frac{n\log n}{n^2\log n}=1/n.
 \eeqax
 Combining these estimates with Lemma \ref{A:properties}(4),(5) and using 
 the Cauchy-Schwartz inequality, 
 \beqax
 \E U_n^2
 &\le& 4[4\E(\vgY_{\vPi}^T[\bar\vA(\sV_n)]\vZ_n)^2 + 4\E(\vgY_{\vPi}^T[\bar\vA(\sV_n)]\vR_{\vPi})^2 + 5/n] \\
 &=& 16 [ \vgY_{\vPi}^T[\bar\vA(\sV_n)]\E (\vZ_n\vZ_n^T)[\bar\vA(\sV_n)]\vgY_{\vPi} + 
 \vgY_{\vPi}^T[\bar\vA(\sV_n)]\E (Y_{\vPi}Y_{\vPi}^T)[\bar\vA(\sV_n)]\vgY_{\vPi}] + 20/n \\
 &\le& C\bar\nu^2[\ind_{\vPi}^T [\bar\vA(\sV_n)]^2\ind_{\vPi} + \ind_{\vPi}^T [\bar\vA(\sV_n)] \textrm{Diag}(\ind_{\vPi}) [\bar\vA(\sV_n)] \ind_{\vPi}] +1/n \\
 &\le& C\bar\nu^2 [\frac{1}{\log n} +\frac{\log n}{n}] +1/n.
 \eeqax
 
 \noindent 
 (3) For each $i\in I$, define $\vv_i:=\arg\max_{\vv\in\vPi} \nu_{\vv,\vPi}$. If we let $\vPi(I):=\{v_i: i\in I\}$, then it is easy to see that $\nu(\sG_n) \ge \underline\nu^2\ind_{\vPi(I)}^T[\bar\vA(\sG_n)]\ind_{\vPi(I)}$. Since $|\vPi(I)|\ge n/2$, we can use Lemma \ref{nlog n bd} and conclude that there is a constant $c\in(0,1)$ such that $\ind_{\vPi(I)}^T[\vA(\sG_n)]\ind_{\vPi(I)} \ge cn\log n$. Combining this with  Lemma \ref{A:properties}(1) we get the desired lower bound.
 
  For the upper bound note that $\nu(\sG_n) \le \bar\nu^2 \ind_{\vPi}^T[\bar\vA(\sV_n)]\ind_{\vPi}$, and use Lemma \ref{A:properties}(1),(3). 
 \end{proof}

\addtocontents{toc}{\protect\setcounter{tocdepth}{0}}
\section*{Acknowledgments}
\corO{The authors thank Hubert Lacoin for the reference to \cite{BL15}.}

\addtocontents{toc}{\protect\setcounter{tocdepth}{1}}
\bibliographystyle{amsplain}
\bibliography{DetectionBib}

\addtocontents{toc}{\protect\setcounter{tocdepth}{0}}

\end{document}